\documentclass[10pt]{amsart}
\usepackage{amsmath,amssymb}
\usepackage[T1]{fontenc}
\usepackage{lmodern,textcomp}
\usepackage{color}
\usepackage{graphicx}
\usepackage{amsthm}
\usepackage[all]{xy}
\usepackage{caption}
\usepackage{enumitem}
\usepackage{graphicx}
\usepackage{ytableau}
\usepackage[english]{babel}

\usepackage[pdfborder={0 0 0}]{hyperref}
\usepackage[capitalize]{cleveref}

\usepackage{longtable}

\def\H{{\mathcal{H}}}
\def\Hfunc{{\mathcal{A}}}
\def\M{{\mathcal{M}}}

\def\oC{\overline{\mathcal{C}}}

\def\oM{\overline{\mathcal{M}}}

\def\oH{\overline{\mathcal{H}}}
\def\cH{{\mathcal{H}}}

\def\RR{\mathbb{R}}
\def\CC{\mathbb{C}}
\def\ZZ{\mathbb{Z}}
\def\NN{\mathbb{N}}
\def\QQ{\mathbb{Q}}
\def\PP{\mathbb{P}}

\def\log{{\rm log}}

\def\oddd{{\rm odd}}
\def\odd{{-}}
\def\even{{\rm even}}
\def\even{{+}}

\def\spin{{\rm spin}}
\def\spin{{\pm}}

\def\Vol{{\rm Vol}}
\def\v{{\rm vol}}
\def\cont{{\rm cont}}
\def\ind{{\rm ind}}

\theoremstyle{definition}
\newtheorem{definition}{Definition}[section]
\newtheorem{notation}[definition]{Notation}
\newtheorem{example}[definition]{Example}
\newtheorem{remark}[definition]{Remark}

\theoremstyle{plain}

\newtheorem{theorem}[definition]{Theorem}

\newtheorem{proposition}[definition]{Proposition}

\newtheorem{lemma}[definition]{Lemma}
\newtheorem{corollary}[definition]{Corollary}

\newcommand{\pdv}[2]{\frac{\partial #1}{\partial #2}}
\newcommand{\OP}{\mathrm{OP}}
\newcommand{\SP}{\mathrm{SP}}
\newcommand{\Sergeev}{\mathrm{Se}}
\newcommand{\Sym}{\mathfrak{S}}
\newcommand{\Spin}{\widetilde{\mathfrak{S}}}
\newcommand{\ASpin}{\widetilde{\mathrm{A}}}
\newcommand{\Cliff}{\mathrm{Cl}}
\newcommand{\SSym}{\mathrm{B}}
\newcommand{\Hur}{\mathrm{Hur}}
\newcommand{\Moller}{\mathrm{M}}
\newcommand{\ii}{\mathrm{i}}

\renewcommand{\Re}{\operatorname{Re}}
\renewcommand{\Im}{\operatorname{Im}}

\newcommand{\step}[1]{\noindent\underline{\textit{#1}}}

\renewcommand{\=}{\: =\: }
\newcommand{\defis}{\: :=\: }
\newcommand{\+}{\,+\,}
\newcommand{\meno}{\,-\,}

\begin{document}

\title{Cylinder counts and spin refinement of area Siegel--Veech constants}
\author{Jan-Willem van Ittersum}
\address{Max-Planck-Institut f\"ur Mathematik, Vivatsgasse 7, 53111 Bonn, Germany}
\email{ittersum@mpim-bonn.mpg.de}
\author{Adrien Sauvaget}
\address{Laboratoire AGM, 2 avenue Adolphe Chauvin, 95300, Cergy-Pontoise, France}
\email{arien.sauvaget@math.cnrs.fr}
\date{\today}

\begin{abstract} We study the area Siegel--Veech constants of components of strata of abelian differentials with even or odd spin parity. We prove that these constants may be computed using either: (I) quasimodular forms, or (II) intersection theory. These results refine the main theorems of \cite{CheMoeZag} and~\cite{CheMoeSauZag} which described the area Siegel--Veech constants of the full strata. Along the proof of (II), we establish a new identity for Siegel--Veech constants of cylinders.
\end{abstract}

\maketitle

\setcounter{tocdepth}{1}
\tableofcontents

%%%%%%%%%%%%%%%%%%%%%%%%%%%%%%%%
%%%%%%%%%%%%%%%%%%%%%%%%%%%%%%%%
\section{Introduction}
%%%%%%%%%%%%%%%%%%%%%%%%%%%%%%%%
%%%%%%%%%%%%%%%%%%%%%%%%%%%%%%%%

%%%%%%%%%%%%%%%%%%%%%%%
\subsection{Relations between cylinder Siegel--Veech constants }
%%%%%%%%%%%%%%%%%%%%%%%

Let $g,n$ be non-negative integers satisfying $2g-2+n{> 0}$. Let $\M_{g,n}$ and $\oM_{g,n}$ be the moduli spaces of smooth and stable complex curves of genus $g$ with $n$ distinct markings. We denote by $p:\oH_{g,n}\to \oM_{g,n}$  the {\em Hodge bundle}, i.e., the vector bundle whose fiber at the point $(C,x=\{x_1,\ldots,x_n\})$ is the space $H^0(C,\omega_C)$ of abelian differentials on $C$. If $\mu=(m_1,\ldots,m_n)$ is a vector of positive integers satisfying 
$$|\mu| \overset{\rm def}{=} \sum_{i=1}^n m_i= 2g-2+n,$$ 
then we denote by $\H(\mu)$ the {\em stratum of abelian differentials of type $\mu$}, i.e., the sub-space of $\oH_{g,n}$ of differentials $(C,x,\eta)$ satisfying: $C$ is smooth, and ${\rm ord}_{x_i}(\eta)=m_i-1$ for all $1\leq i\leq n$. This space is equipped with a canonical measure $\nu$, called the {\em Masur--Veech measure}. If $X$ is a component of $\H(\mu)$, then we denote by
$$
\Vol(X)=(4g-2+2n)\times \nu\left\{ (C,x,\eta)\in X, \text{ s.t.\ } \frac{\ii}{2} \int_C \eta\wedge \bar{\eta} <1 \right\},
$$
the 
{\em Masur--Veech}  volume of $X$. This volume is finite and rational up to a power of~$\pi$ by \cite{Mas}, \cite{Vee}, \cite{EskOko}, and~\cite{EskOkoPan}. 

Let $(C,x,\eta)$ be a differential of type~$\mu$. The differential $\eta$ defines a flat metric with trivial holonomy on $C\setminus\{x_1,\ldots,x_n\}.$  Each $x_i$ is a conical singularity of this metric with angle $m_i(2\pi)$. The union of closed geodesics of a given homotopy type in this open surface forms a cylinder $Z$ whose {\em width} $w(Z)$ is the length of any geodesics in the homotopy class. This cylinder is bounded on each side by at least one singularity%, and exactly one for almost all  surfaces in~$\H(\mu)$
. We denote
\begin{eqnarray*}
\mathcal{N}(C,\eta,L)_0 &=&  \sum_{Z \text{ s.t.\ } w(Z)< L} \frac{{\rm area}(Z)}{{\rm area}({C})}, \\
 \mathcal{N}(C,\eta,L)_{cyl} &=&  \sum_{Z \text{ s.t.\ } w(Z)< L} \frac{1}{{\rm area}({C})}.
\end{eqnarray*}
Here we consider a refinement of the count of cylinders: for all $1\leq i\leq n$, we denote 
\begin{eqnarray*}
\mathcal{N}(C,\eta,L)_{i} &=& \frac{1}{2} \sum_{Z \text{ s.t.\ } w(Z)< L} \frac{f_i(Z)}{{\rm area}({C})},
\end{eqnarray*}
where $f_i(Z)=0, 1,$ or $2$ if $x_i$  bounds $Z$ on $0, 1,$ or $2$ sides. 
For a connected component $X$ of $\H(\mu)$, there exist constants $c_0(X)$, $c_{cyl}(X)$, $c_1(X),\ldots, c_n(X)$  satisfying
\begin{eqnarray*}
&\mathcal{N}(C,\eta,L)_i\underset{L\to \infty}{\sim} c_i(X) \cdot \pi L^2
\end{eqnarray*}
for almost all abelian differentials in $X$~\cite{Vee1,EskMasZor}. These constants are elements of $\pi^{-2} \QQ$ and may be expressed as ratios between volumes of connected components of strata (see Section~\ref{sec:chains}). In~\cite{Vor} Vorobets showed the following identity holds:
$$c_{cyl}(X)=(2g-2+n)c_0(X).$$
 Our first theorem is a refinement of this identity.
\begin{theorem}\label{th:cylinders}
For all $1\leq i\leq n$, we have $c_i(X)=m_i\cdot c_0(X).$ 
\end{theorem}

\begin{example}
A standard class of abelian differentials is provided by coverings of the square torus ramified above a single point -- {\em square-tiled surfaces}.  The differential on the covering surface is the pull-back of the unique (up to scalar) differential on the torus, and its singularities are determined by the orders of ramification. The polygon of Figure~\ref{fig:dessin} is a square-tiled surface in~$\cH(4,2)$: the surface is obtained by identifying the edges of the polygon with the same labels. The red/blue vertices of the polygon are identified along this process to produce conical singularities of angles $8\pi$ and $4\pi$ respectively. In this example, the red cylinder is bounded twice by the red singularity, while the green one is bounded once by each singularity. 

 \begin{figure}[ht]
 \vspace{-5pt}
\includegraphics[scale=0.28]{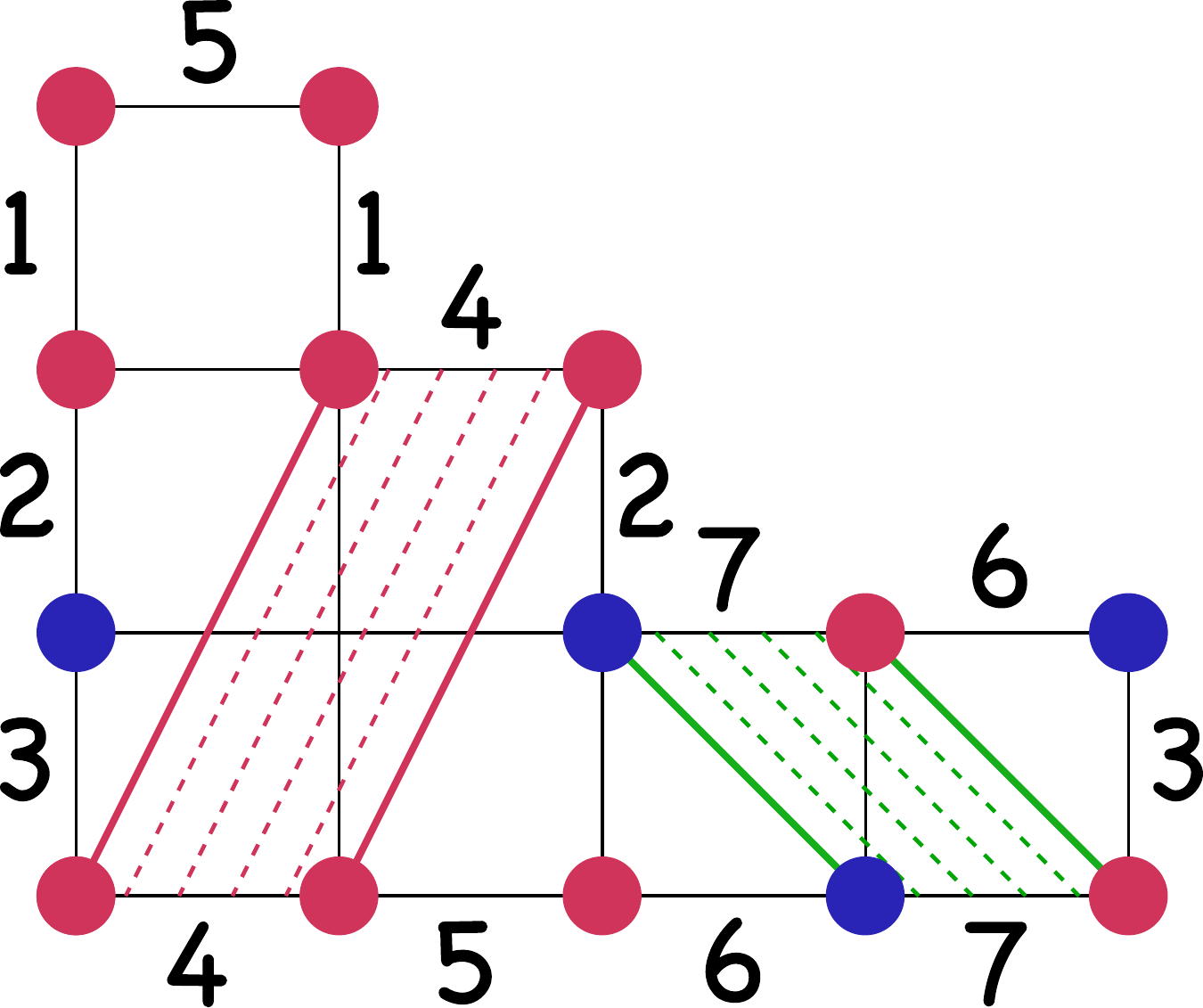}
\vspace{-10pt}
\caption{\label{fig:dessin} Square-tiled surface in $\cH(4,2)$. }
\end{figure}
\end{example}
Theorem~\ref{th:cylinders} implies that if we choose a random boundary of a cylinder of large width for a generic deformation of this square-tiled surface, then the probability that it contains the red singularity is approximately $\frac{2}{3}$.

%%%%%%%%%%%%%%%%%%%%%%%
\subsection{Spin parity of abelian differentials}
%%%%%%%%%%%%%%%%%%%%%%%

Here we assume that the entries of~$\mu$ are all odd. An element $(C,x,\eta)$ in $\H(\mu)$ determines a canonical {\em spin structure}, i.e.\ a line bundle $L\to C$ such that $L^{\otimes 2}\simeq \omega_C$.  This line bundle is defined as $L=\mathcal{O}_C\left(\frac{m_1-1}{2} x_1 + \ldots + \frac{m_n-1}{2} x_n \right)$. The {\em sign}---or {\em Arf invariant}---of an abelian differential in $\mathcal{H}(\mu)$ equals %the sign of 
$(-1)^{h^0(C,L)}$. By classical results of Mumford and Atiyah, this sign is constant in connected families of spin structures~\cite{Mum3,Ati}. Thus we denote by $\H(\mu)^\even/\H(\mu)^\odd$ the components of $\H(\mu)$ of even/odd differentials.

\begin{remark} %The space $\H(\mu)$ may have up to 3 connected components. 
{The components $\H(\mu)^\even$ and $\H(\mu)^\odd$ may be disconnected.}
Indeed, when $\mu=(2g-1)$ or $(g,g)$, then $\H(\mu)$ contains a connected component of abelian differentials supported on hyperelliptic curves. %$\H(\mu)^\hyp$
This connected component may be included in the even or odd component depending on the value of $g$ (see~\cite{KonZor}). {Hence, the space $\H(\mu)$ may have up to 3 connected components.}
As Theorem~\ref{th:cylinders} holds trivially on hyperelliptic components, %we will not need to study these separately.
{it suffices to study $c_i(X)$ for $X=\H(\mu)^\even$ and $X=\H(\mu)^\odd$.}
\end{remark} 

If $X$ is a union of components of $\H(\mu)$, and $0\leq i \leq n$, then $c_i(X)$ stands for the average of the Siegel--Veech constants of its connected components (weighted by the Masur--Veech volume). We denote by $c_i(\mu)$ the Siegel--Veech constants of~$\cH(\mu)$. Moreover, we write
\begin{align*}
\Vol(\mu) &= \Vol(\H(\mu)), & \Vol^\spin(\mu)&=\Vol\left(\cH(\mu)^\even\right)-\Vol\left(\cH(\mu)^\odd\right),\\
c_i(\mu) &= \Vol(\H(\mu)), & c_i^\spin(\mu)&=c_i\left(\cH(\mu)^\even\right)-c_i\left(\cH(\mu)^\odd\right).
%c_i^\spin(\mu)&=& c_i\left(\cH(\mu)^\even\right) \cdot \frac{\Vol\left(\cH(\mu)^\even\right)}{\Vol\left(\cH(\mu)\right)}-c_i\left(\cH(\mu)^\odd\right) \cdot \frac{\Vol\left(\cH(\mu)^\odd\right)}{\Vol\left(\cH(\mu)\right)}.
\end{align*}
 The functions $\Vol,c_0,$ and $\Vol^\spin$ may be expressed either:
\begin{enumerate}
\item[(I)] As the value $q\to 1$ of the $q$-expansion of quasimodular forms (see~\cite{EskOko}, \cite{CheMoeZag}, and~\cite{EskOkoPan});
\item[(II)] As intersection numbers on~$\PP\oH(\mu)$, the Zariski closure of the projectivization of $\H(\mu)$
 in the projectivized Hodge bundle (see~\cite{CheMoeSauZag}).
  \end{enumerate} 
We extend these two results to the function $c_0^\spin$ in Theorems~\ref{th:withqmf} and~\ref{th:intspin} respectively. 
\begin{remark}
Several results for the function~$c_0$ may be transposed to~$c_0^\spin$ with parallel arguments. However, certain ingredients were missing in the previous works to obtain the complete description of~$c_0^\pm$. We emphasize two arguments that play an important role:
\begin{itemize}[leftmargin=25pt]
\item The quasimodular forms approach relies on the description of the character table of the Sergreev group (Propositions~\ref{prop:charSergeev} and~\ref{prop:charSergeev0}), while the computation of $\Vol^\spin$ only relied on the description of the irreducible spin \emph{super}representations. As a result, the expression of the differential operator~$\partial_2$ appearing in Theorem~\ref{th:withqmf} is different from the conjectural expression of~\cite[Section~10.3]{CheMoeSauZag}.
\item The geometric counterpart relies on the original description of cylinder configurations by Eskin--Masur--Zorich. We show that the expression of the area Siegel--Veech constants as intersection numbers is essentially equivalent to the statement of Theorem~\ref{th:cylinders} above. It is interesting to remark that this approach also recovers the result of~\cite{CheMoeSauZag} on $c_0$ without using quasimodular forms.  
 \end{itemize}
 \end{remark}
 
\subsection*{Acknowledgements} The authors thank Elise Goujard, Martin M\"oller, Don Zagier, and Anton Zorich for interesting discussions on the topic. The second author is especially grateful to Dawei Chen for explaining the correspondence between cylinder configurations of Eskin, Masur, and Zorich and twisted graphs. This project grew out of discussions on spin Hurwitz numbers started in a reading group proposed by Martijn Kool at Utrecht University.

%%%%%%%%%%%%%%%%%%%%%%%%%%%%%%%%
%%%%%%%%%%%%%%%%%%%%%%%%%%%%%%%%
\section{Weighted spin Hurwitz numbers and quasimodular forms}
%%%%%%%%%%%%%%%%%%%%%%%%%%%%%%%%
%%%%%%%%%%%%%%%%%%%%%%%%%%%%%%%%

We compute $c_0^\spin$ as the limit of weighted spin Hurwitz numbers of the torus of large degree. By the work of \cite{EskOkoPan} generating series of spin Hurwitz numbers can be expressed in terms of quasimodular forms. We extend the approach of \cite{CheMoeZag} for $c_0$ to compute $c_0^\spin$ as the limiting value $q\to 1$ in the $q$-expansion of a quasimodular form, or, more precisely, as the leading term of the growth polynomial associated to a quasimodular form. We proceed as follows:
\begin{itemize}[leftmargin=25pt]
\item In \cref{sec:spin} we introduce the spin $q$-bracket and the symmetric functions~$p_k$. These functions serve as an analogue of both the shifted symmetric functions~$Q_k$ and the hook-length moments~$T_p\mspace{1mu}$. In particular, its brackets are quasimodular forms and the action of the $p_k$ inside strict brackets can be encoded by differential operators.
\item In \cref{sec:growth} we recall the growth polynomials of quasimodular forms
and relate this growth to the aforementioned differential operators. 
\item In \cref{sec:Sergeev} we recall the representation theory of the spin symmetric group, and explain the representation theory of the so-called Sergeev group.
\item In \cref{sec:spinHurwitz} we refine the aforementioned result of \cite{EskOkoPan}. More precisely, we compute the generating series of \emph{weighted} spin Hurwitz numbers. Here, we extensively make use of the character table of the Sergeev group. We find that strict brackets of the $p_k$ compute weighted Hurwitz numbers. 
\item In \cref{sec:recursion} we take all these results together to prove the evaluation of $c_0^\spin(\mu)$ as the value $q=1$ of the $q$-expansion of quasimodular forms.
\end{itemize}

%%%%%%%%%%%%%%%%%%%%%%%
\subsection{Spin bracket and quasimodular forms}\label{sec:spin}
%%%%%%%%%%%%%%%%%%%%%%%

Denote by $\SP$ the set of all \emph{strict} partitions of integers (i.e., partitions where all part sizes are different) and by $\OP$ the set of all \emph{odd} partitions (i.e., partitions where all part sizes are odd). For $f:\SP\to \CC$, denote the \emph{spin $q$-bracket} $\langle f \rangle\in \CC[\![q]\!]$ by
\[ \langle f \rangle \defis 
\frac{\sum_{\lambda \in \SP} \,(-1)^{\ell(\lambda)}\, f(\lambda) \, q^{|\lambda|}}
{\sum_{\lambda \in \SP}(-1)^{\ell(\lambda)}\,q^{|\lambda|}} ,\]
where $|\lambda|:=\sum_{i}\lambda_i$ denotes the size of $\lambda$. The denominator is given by
\[\sum_{\lambda \in \SP}(-1)^{\ell(\lambda)}\,q^{|\lambda|}\=\prod_{m\geq 1}(1-q^m).\]

For a partition $\lambda$ and a positive integer $m$, we denote by $r_m(\lambda)$ the number of parts of $\lambda$ equal to $m$. Observe that for strict partitions $\lambda$ we have $r_m(\lambda)\in \{0,1\}$.  Moreover, we let
 \begin{align*}
 p_k(\lambda)&\defis \sum_{i}\lambda_i^k \= \sum_{m=1}^\infty m^k\,r_m(\lambda), && (k\in \ZZ) \\
 \mathbf{p}_k &\defis -\tfrac{1}{2}\zeta(-k)\+ p_k && (k\in \ZZ_{\geq 0})
 \end{align*}
 be the \emph{symmetric power sums} (with an additional constant). Note that this constant equals $-\tfrac{1}{2}\zeta(-k) = \frac{B_{k+1}}{2(k+1)},$ with $B_{k+1}$ the $(k+1)$th Bernoulli number. Write $\Lambda=\CC[\mathbf{p}_1,\mathbf{p}_3,\mathbf{p}_5,\ldots]$ for the \emph{symmetric algebra} and assign to $\mathbf{p}_i$ weight $i+1$. Then, for all $f\in \Lambda,$ the spin bracket $\langle f\rangle$ is known to be a quasimodular form \cite[Section~3.2.2]{EskOkoPan}, which is made precise in the following result. 
 Recall that the \emph{Eisenstein series}~$G_k$, given by
 \begin{align}\label{eq:eis} G_k := -\frac{B_k}{2k}+\sum_{m,r\geq 1} m^{k-1} q^{mr} \qquad (k\geq 2 \text{ even}),\end{align}
is an example of a quasimodular form of weight~$k$, and that every quasimodular form is a polynomial in these series: the space of quasimodular forms~$\widetilde{M}$ (for the full modular group $\mathrm{SL}_2(\ZZ)$) is given by $\widetilde{M} =\QQ[G_2,G_4,G_6].$  For an introduction to quasimodular forms, see \cite{Zag}.

In order to state this result we introduce \emph{connected} brackets, following \cite[Section~11]{CheMoeZag}. The \emph{connected spin $q$-bracket} is the multilinear map $(\CC^\SP)^{\otimes n} \to \CC[\![q]\!]$, defined by
\begin{align}\label{eq:connectedq}\langle\, f_1 \mid \cdots \mid f_n\,\rangle \defis \sum_{\alpha\in \Pi(n)}\mu(\alpha)\prod_{A\in \alpha}\biggl\langle \prod_{a\in A} f_a\biggr\rangle,\end{align}
where $\Pi(n)$ is the set of set partitions of $[\![1,n]\!]:=\{1,\ldots,n\}$, and $\mu$ is the corresponding M\"obius function $\mu(\alpha) = (-1)^{\ell(\alpha)-1}(\ell(\alpha)-1)!$ with $\ell(\alpha)$ the length (cardinality) of the partition~$\alpha$.  Let $u_1,u_3,\ldots$ be formal variables and for $\rho \in \OP$, write $u_\rho = \prod_{i} u_{\rho_i}$. Similarly, write $\mathbf{p}_\rho=\prod_{i} \mathbf{p}_{\rho_i}$. We introduce the generating series $\Psi$ (and $\Psi^\circ$) of (connected) brackets of symmetric power sums by
\begin{align*}
  \Psi(u_1,u_3,\ldots) &\defis \Bigl\langle \exp\Bigl(\sum_{k}   \mathbf{p}_k\, u_k\Bigr)\Bigr\rangle \= \sum_{n \geq 0} \frac{1}{n!} \sum_{\ell_1,\ldots,\ell_n} \langle \mathbf{p}_{\ell_1}  \cdots  \mathbf{p}_{\ell_n}\rangle \,  u_{\ell_1}\cdots u_{\ell_n} \,, \\
%\end{align*}
%where the sum is over all \emph{odd} integers $k$, and $\ell_1,\ldots,\ell_n$ respectively. 
%The generating series of connected spin $q$-brackets of the generators~$\mathbf{p}_k$ %is the multilinear map $\Lambda^{\otimes n} \to \widetilde{M}$, defined in terms of the generating series
%\begin{align*}
 \Psi^\circ(u_1,u_3,\ldots) &\defis \sum_{n > 0} \frac{1}{n!} \sum_{\ell_1,\ldots,\ell_n} \langle \mathbf{p}_{\ell_1} | \cdots | \mathbf{p}_{\ell_n} \rangle \, u_{\ell_1} \cdots u_{\ell_n} \,,
 \end{align*}
 where the sum is over all \emph{odd} integers $k$, and \emph{odd} $\ell_1,\ldots,\ell_n$ respectively. Then, by the properties of the connected bracket we have
$\exp \Psi^{\circ} = \Psi,$
 and in this notation the aforementioned result of \cite{EskOkoPan} is as follows.
\begin{proposition}\label{thm:quasimodular} We have
\begin{align*}
% \Psi(u_1,u_3,\ldots) &\= \exp\Bigl(\sum_{\rho \in \OP} D^{\ell(\rho)-1} G_{|\rho|-\ell(\rho)+2}\: u_\rho\Bigr), \\
 \Psi^\circ(u_1,u_3,\ldots) &\= -\sum_{\substack{\rho \in \OP\\\rho\neq \emptyset}} D^{\ell(\rho)-1} G_{|\rho|-\ell(\rho)+2}\: u_\rho \,,
 \end{align*}
where $D=q\pdv{}{q}$ and $G_k$ is the Eisenstein series of weight~$k$ {\upshape(}see \eqref{eq:eis}{\upshape)}. % and the first sum is over all positive, odd integers~$k$. 
\end{proposition}
Note that this result determines  $\langle f\rangle$ for all $f\in \Lambda$. In particular, if $f$ is of weight~$k$, the spin bracket~$\langle f \rangle$ is a quasimodular form of weight~$k$. 

Two consequences of this result will be important for us: (i) a recursive formula for spin brackets, and (ii) the definition of a modified $q$-bracket for functions of the form $p_{-1}f$ with $f\in\Lambda$. This recursive formula should be compared with the recursive formula for the hook-length moments in \cite[Theorem~16.1]{CheMoeZag} and to a similar result for symmetric functions in the non-spin setting in \cite[Proposition 6.2.1]{vISymmetric}. 

\begin{corollary}\label{cor:pkf} For all $i,j\geq 0$ with $i$ even, the differential operators $\varrho_{i,j}:\Lambda\to\Lambda$ given by
\[ \varrho_{i,j} \= 
\sum_{\substack{\rho\in \OP\\\ell(\rho)=j\\|\rho|=i+j}} \frac{1}{\mathrm{Aut}(\rho)} \pdv{^{\ell(\rho)}}{\mathbf{p}_\rho} \qquad \qquad \Bigl(\mathrm{Aut}(\rho)\=\prod_{m\geq 1} r_m(\rho)!\Bigr),\]
are such that for all odd $k\geq 1$ and $f\in \Lambda$ one has
\[ \langle \mathbf{p}_k f\rangle \= -\sum_{i,j\geq 0} \langle\varrho_{i,j}(f)\rangle \, D^{j} G_{k+i+1}.\]
\end{corollary}
\begin{proof}
Without loss of generality, we assume that $f$ is a monomial. Then, by extracting the coefficient of $u_\nu$ in \cref{thm:quasimodular} for some odd partition $\nu$ containing a part equal to $k$, the result follows.
\end{proof}
Note that
\[\varrho_{i,0}=\delta_{i,0} \cdot \mathrm{Id} \qquad \text{and} \qquad \varrho_{0,j}=\frac{1}{j!}\pdv{^j}{\mathbf{p}_1^j} \qquad (i,j\geq 0).\]

Following the proof of \cite[Theorem~16.1]{CheMoeZag} in the spin setting, we deduce that a certain linear combination of brackets involving $p_{-1}$ is quasimodular. 
\begin{corollary}\label{thm:*} For all $f\in \Lambda_k$, \emph{the modified spin $q$-bracket}
\[ \langle f \rangle^* \defis \langle p_{-1}f\rangle - \langle p_{-1}\rangle \langle f\rangle - \frac{1}{24}\langle \partial_2(f)\rangle \]
is a quasimodular form of weight~$k$, where
$\partial_2 = \varrho_{0,1} = \pdv{}{\mathbf{p}_1}.$ More precisely, we have
\[\langle f \rangle^* = \sum_{i\geq 2,j\geq 0} \langle\varrho_{i,j}^*(f)\rangle \, D^j G_i\,\]
where $\rho_{i,j}^*=\rho_{i,j}+\delta_{i,2}\rho_{0,j+1}$ and $D=q\pdv{}{q}$.
\end{corollary}
\begin{proof}
By the previous lemma, and as $\rho_{i,0}=\delta_{i,0}\cdot\mathrm{Id}$, we find that $\langle p_{-1}f \rangle$ equals
\begin{align*}
%\langle p_{-1}f \rangle &\= 
&\qquad\sum_{i,j} \langle\varrho_{i,j}(f)\rangle \, \Bigl(\sum_{m,r} m^{i+j-1} r^j q^{mr}\Bigr) \\
&\= \langle f\rangle \Bigl(\sum_{m,r}m^{-1} q^{mr}\Bigr) \+  \sum_{j\geq 1} \langle \varrho_{0,j}(f)\rangle D^{j-1}\biggl(\frac{1}{24}+G_2 \biggr)\+\!\!\sum_{i\geq 2,j\geq 1} \!\!\langle\varrho_{i,j}(f)\rangle \, D^j G_i\\
&\=\langle f\rangle \langle p_{-1}\rangle\+\frac{1}{24}\langle \partial_2(f)\rangle\+ \sum_{j\geq 1} \langle \varrho_{0,j}(f)\rangle D^{j-1}G_2\+\!\!\sum_{i\geq 2,j\geq 1} \langle\varrho_{i,j}(f)\rangle \, D^j G_i\,. 
\end{align*}
As  $\langle \varrho_{i,j}(f)\rangle$ is quasimodular of weight $k-i-2j$ the result follows.
\end{proof}

%%%%%%%%%%%%%%%%%%%%%%%
\subsection{Growth polynomials of quasimodular forms}\label{sec:growth}
%%%%%%%%%%%%%%%%%%%%%%%
Following \cite[Section~9]{CheMoeZag}, there exists a unique algebra homomorphism $\mathrm{ev}:\widetilde{M}\to \QQ[\pi^2][1/h]$ such that
\[ F(\tau) \= \mathrm{ev}[F](h) + O(e^{-h}) \qquad (q=e^{2\pi i\tau}=e^{-h})\]
as $h \to 0$ (i.e., $q\to 1$). %Note that $\mathrm{ev}[F](h)$ is called the \emph{growth polynomial} of~$F$---it is a polynomial in $1/h$.
We call $\mathrm{ev}[F](h)$ the \emph{growth polynomial} of~$F$---it is a polynomial in $1/h$; for more details see the aforementioned paper by Chen, M\"oller and Zagier. 
In particular, this morphism is characterized by the following three properties:
\begin{enumerate}[label=(\roman*)]\itemsep3pt
\item $\displaystyle\mathrm{ev}[F](h) \= a_0(f)\, \Bigl(\frac{2\pi \ii}{h}\Bigr)^{k}$ for $F\in M_{k}$\,,
\item $\displaystyle\mathrm{ev}[G_2](h) \= \frac{\zeta(2)}{h^2}-\frac{1}{2h}$\,,
\item $\displaystyle\mathrm{ev}[DF](h) \= -\pdv{}{h}\,\mathrm{ev}[F]$ for $F\in \widetilde{M}_{k}$\,.
\end{enumerate}

We will be interested in the leading coefficient of the growth polynomial. For $f_1,\ldots,f_r\in \Lambda$, we define the \emph{$\hbar$-bracket} and its \emph{leading term} by
\begin{align*}
\langle\, f_1\mid\cdots \mid f_r\,\rangle_{\hbar} &\defis  \mathrm{ev}[\langle\, f_1\mid\cdots\mid f_r \,\rangle](\hbar),\\
%\langle\, f_1\mid\cdots \mid f_r\,\rangle_{L} &\defis \lim_{\hbar \to 0} \hbar^{k-r+1}  \langle\, f_1\mid\cdots\mid f_r\,\rangle_{\hbar}, 
\langle\, f_1\mid\cdots \mid f_r\,\rangle_{L} &\defis \lim_{\hbar \to 0} \frac{\hbar^{k-r+1}}{(2\pi\ii)^{k-2r+2}}  \langle\, f_1\mid\cdots\mid f_r\,\rangle_{\hbar}\,, 
\end{align*}
where $k$ is the sum of the weights of the $f_i$. Note that by \cite[Proposition 11.1]{CheMoeZag} this limit is well-defined. By \cref{thm:*} we can extend the notation and allow an insertion of $p_{-1}$:
\begin{align}
\langle\, p_{-1} \mid f_1\mid\cdots \mid f_r\,\rangle_{\hbar} &\defis  \mathrm{ev}[\langle\, p_{-1} \mid f_1\mid\cdots\mid f_r \,\rangle](\hbar),\\
\langle\, p_{-1} \mid f_1\mid\cdots \mid f_r\,\rangle_{L} &\defis \lim_{\hbar \to 0} \frac{\hbar^{k-r+1}}{(2\pi\ii)^{k-2r+2}}  \langle\, p_{-1} \mid f_1\mid\cdots\mid f_r\,\rangle_{\hbar}\,. \label{eq:L}
\end{align}
The behaviour as $\hbar\to 0$ also determines the growth of the first $N$ Fourier coefficients. That is, the \emph{$N$-bracket}, which we define by
\[ [\,f_1 \mid \cdots \mid f_r\,]_N \defis \sum_{n=1}^N a_n(f_1,\ldots,f_r),\]
where we wrote $ \langle f_1 | \cdots | f_r\rangle \= \sum_{n\geq 0} a_n(f_1,\ldots,f_r) \,q^n$, admits the following growth \cite[Proposition~9.4]{CheMoeZag}.
\begin{proposition}\label{prop:growthN} 
For $f_1,\ldots,f_r\in \Lambda$ of weights $k_1,\ldots,k_r$ with $k=\sum k_i$ we have
\begin{align*}  
[ f_1 \mid \cdots \mid f_r]_N &\= \langle\, f_1\mid \cdots \mid f_r\,\rangle_L \,\frac{N^{k-r+1}(2\pi\ii)^{k-2r+2}}{(k-r+1)! } \+ O\bigl(N^{k-r}\log N\bigr) \\
[p_{-1} \mid f_1 \mid \cdots \mid f_r]_N &\= \langle\,p_{-1} \mid f_1\mid \cdots \mid f_r \,\rangle_L \,\frac{N^{k-r+1}(2\pi\ii)^{k-2r+2} }{(k-r+1)!} \+ O\bigl(N^{k-r}\log N\bigr)
\end{align*}
as $N\to \infty$. 
\end{proposition}

In the rest of this subsection, we will now state two lemmas we need in the sequel. First of all, as a corollary of \cref{thm:*}, we determine the leading terms of $p_{-1}$ insertions. Secondly, we discuss the relationship between growth polynomials and differential operators. 

%It turns out to be more convenient to work with a different generating set of functions. 
For all $\ell\geq 1$ odd, we define:
\begin{align}\label{def:h}
\mathbf{h}_\ell \defis \frac{-1}{2\ell} \left[u^{\ell+1}\right]  \mathbf{P}(u)^\ell, \quad \text{where } \mathbf{P}(u)\defis \exp\Bigl(-\sum_{\substack{k\geq 0 \\ k\, {\rm odd}}}  2\mathbf{p}_{k}\, u^{k+1}\Bigr).
\end{align}
%These functions are homogeneous, and $\mathbf{f}_\ell-\ell \mathbf{h}_\ell$ is of weight less than $\ell+1$. 
These functions are the highest weight part of the central characters $\mathbf{f}_\ell$ introduced in~\eqref{eq:centralchar}, i.e., for all odd $\ell\geq 1$ the function $\mathbf{h}_\ell$ is homogeneous, and $\mathbf{f}_\ell- \mathbf{h}_\ell/\ell$ is of weight less than $\ell+1$ \cite[Theorem~6.7]{CheMoeSauZag}.
We will be interested in the growth of the coefficients in the following sequences
\begin{align*}
  \Psi^H(u_1,u_3,\ldots) &\defis \sum_{n \geq 0} \frac{1}{n!} \sum_{\ell_1,\ldots,\ell_n} \langle \mathbf{h}_{\ell_1}  \cdots  \mathbf{h}_{\ell_n}\rangle \,  u_{\ell_1}\cdots u_{\ell_n} \,, \\
 \Psi^{H,\circ}(u_1,u_3,\ldots) &\defis \sum_{n \geq  0} \frac{1}{n!} \sum_{\ell_1,\ldots,\ell_n} \langle \mathbf{h}_{\ell_1} | \cdots | \mathbf{h}_{\ell_n} \rangle \, u_{\ell_1} \cdots u_{\ell_n} \,, \\
%\Psi^H_{-1}(u_1,u_3,\ldots) &\defis\sum_{n \geq 0} \frac{1}{n!} \sum_{\ell_1,\ldots,\ell_n} \langle p_{-1}\mathbf{h}_{\ell_1}  \cdots  \mathbf{h}_{\ell_n}\rangle \,  u_{\ell_1}\cdots u_{\ell_n} \,, \\
 \Psi^{H,\circ}_{-1}(u_1,u_3,\ldots) &\defis \sum_{n \geq 0} \frac{1}{n!} \sum_{\ell_1,\ldots,\ell_n} \langle p_{-1} | \mathbf{h}_{\ell_1} | \cdots | \mathbf{h}_{\ell_n} \rangle \, u_{\ell_1} \cdots u_{\ell_n}  \,,  \\
  \mathcal{C}^{H,\circ}_{-1}(u_1,u_3,\ldots) &\defis \frac{-1}{24 \Psi^H}\sum_{n \geq 0} \frac{1}{n!} \sum_{\ell_1,\ldots,\ell_n}\sum_{j\geq 1} \langle\partial_2^j(\mathbf{h}_{\ell_1} \cdots \mathbf{h}_{\ell_n})\rangle\, u_{\ell_1} \cdots u_{\ell_n}\,.
 \end{align*}

By \cref{thm:*}, in the spin setting \cite[Lemma~10.4]{CheMoeSauZag} reads
\begin{corollary}\label{lem:L}
The leading terms of 
$\Psi^{H,\circ}_{-1}$ and $\mathcal{C}^{H,\circ}_{-1}$
agree.
\end{corollary}

Let the differential operator $\mathcal{D}:\Lambda\to \Lambda$ be given by
\[2\mathcal{D} = -\pdv{}{\mathbf{p}_1}+\sum_{\ell_1,\ell_2\geq 1} (\ell_1+\ell_2)\mathbf{p}_{\ell_1+\ell_2-1}\pdv{^2}{\mathbf{p}_{\ell_1}\,\partial\mathbf{p}_{\ell_2}}.\]
In \cite[Proof of Proposition~6.10]{CheMoeSauZag} it was shown that (we correct their formula by a factor~$\tfrac{1}{2}$)
\[\mathfrak{d} \langle f \rangle = \langle \mathcal{D}(f) \rangle \]
for all $f\in \Lambda$, where $\mathfrak{d}$ is the unique derivation on quasimodular forms given by $\mathfrak{d}(G_2) = -\frac{1}{2}$ and $\mathfrak{d}(f)=0$ if~$f$ is modular. 
This operator~$\mathcal{D}$ is extremely useful in determining the growth of the coefficients of $F=\langle f \rangle_q$ for $f\in \Lambda$. 
Namely, by 
%\begin{proposition}[
{\cite[Proposition~6.10]{CheMoeSauZag}},
%] 
for all $f\in \Lambda_k$ we have
\begin{align} 
\langle f\rangle_{\hbar} \= \frac{(2\pi \ii)^k}{\hbar^k} (e^{(2\pi \ii)^{-2}\hbar \mathcal{D}} f)(\emptyset). \label{eq:hbracexpr}
 \end{align}
%\end{proposition}
(Be careful that the definition of the $\hbar$-bracket in \cite[Eq.~34]{CheMoeSauZag} differs by a power of $2\pi\ii$ of the $\hbar$-bracket in \cite{CheMoeZag} and in this work.) Observe that the evaluation at the partition~$\emptyset$ of~$0$ is explicitly given by $\mathbf{p}_k(\emptyset) = -\frac{1}{2}\zeta(-k)$. 

Later, we will make use of the following commutation relation, which is the spin analogue of the commutation relation in \cite[Lemma 10.5]{CheMoeSauZag}. 
\begin{lemma}\label{lem:comm} The commutation relation
\[\partial_2 \circ e^\mathcal{D}(f) \= e^\mathcal{D} \sum_{j\geq 1} \partial_2^j(f) \qquad \Bigl(\partial_2 \= \pdv{}{\mathbf{p}_1}\Bigr)\]
holds for all $f\in \Lambda$. 
\end{lemma}
\begin{proof}
First, observe that $[\partial_2,\mathcal{D}]=\partial_2^2$. Hence, by induction we find 
\[ [\partial_2^i,\mathcal{D}]\=i\,\partial_2^{i+1} \qquad (i\geq 0).\]
Next, again by induction, we show that
\[\partial_2\mathcal{D}^i \= \sum_{k=0}^i \frac{i!}{k!} \mathcal{D}^k\partial_2^{i-k+1}\]
for all $i\geq 0$. For $i=0$ this is trivial. By assuming the result for $i=j$, we obtain
\begin{align*}
\partial_2 \mathcal{D}^{j+1} &\= \sum_{k=0}^j \frac{j!}{k!} \mathcal{D}^k\partial_2^{j-k+1}\mathcal{D} \\
&\= \sum_{k=0}^j \frac{j!}{k!} \bigl(\mathcal{D}^{k+1}\partial_2^{j-k+1}\+(j-k+1)\,\mathcal{D}^{k}\partial_2^{j-k+2}\bigr) \\
&= \sum_{k=1}^{j+1} \frac{j!k}{k!} \mathcal{D}^{k}\partial_2^{j-k+2}\+ \sum_{k=0}^{j} \frac{j!(j-k+1)}{k!} \mathcal{D}^{k}\partial_2^{j-k+2}\\
&= \sum_{k=0}^{j+1} \frac{(j+1)!}{k!} \mathcal{D}^{k}\partial_2^{(j+1)-k+1},
\end{align*}
proving the claim. 
We conclude that
\[
\partial_2 \circ e^\mathcal{D} \= \sum_{i\geq 0} \frac{\partial_2 \mathcal{D}^i}{i!} 
\= \sum_{i\geq 0}\sum_{k=0}^i \frac{1}{k!}\mathcal{D}^k\partial_2^{i-k+1} 
\= e^\mathcal{D} \sum_{j\geq 1}\partial_2^j 
%\= e^\mathcal{D} \Bigl(\sum_{j\geq 1}j!\,\varrho_{0,j}\Bigr)
\,. \qedhere\]
\end{proof}

%%%%%%%%%%%%%%%%%%%%%%%%
%\subsection{A differential operators on symmetric functions}
%%%%%%%%%%%%%%%%%%%%%%%%

%%%%%%%%%%%%%%%%%%%%%%%%
\subsection{Representations of the Sergeev group}\label{sec:Sergeev}
%%%%%%%%%%%%%%%%%%%%%%%%
The Siegel--Veech constants are computed as the limit of (weighted) Hurwitz numbers of the torus of large degree. These Hurwitz numbers can be expressed in terms of central characters of the symmetric group. Analogously, spin Siegel--Veech constants and spin Hurwitz numbers can be expressed in terms of central characters of the spin-symmetric group \cite{EskOkoPan, Gun, Lee}. More precisely, spin Siegel--Veech constants can be expressed in terms of central characters corresponding to representations of the Sergeev group, which is closely related to the spin symmetric group. Following \cite{Mor,HofHum,Iva,GiaKraLew}, we recall some results about the representation theory of both groups and explain how they are related. Though most of the results are not new, the character table of the Sergeev group seems not to be available in the literature. We explain how to derive this table from the known results in the literature. 

The spin symmetric group $\widetilde{\mathfrak{S}}_d$ is one of the two representation groups of the symmetric group (for $d\geq 4$), meaning that every projective representation of the symmetric group~$\mathfrak{S}_d$ lifts to a (linear) representation of~$\widetilde{\mathfrak{S}}_d$. Explicitly, it is defined by the central extension
\[ 0 \to \ZZ/2\ZZ \to \Spin_d \xrightarrow{\pi} \Sym_d \to 1\]
and can be presented by
\[\Spin_d = \langle t_1,\ldots,t_{d-1}, \varepsilon \mid \varepsilon^2=1, t_j^2=\varepsilon, (t_jt_{j+1})^3=\varepsilon, (t_jt_k)^2=\varepsilon \text{ for } |j-k|\geq 2\rangle.\]
The projection $\pi$ to $\Sym_d$ is given by sending $\varepsilon$ to the neutral element and $t_j$ to the transposition $(j, j+1)$. 

Note that the element~$\varepsilon$ is central. Hence, $\varepsilon$ acts by $\pm 1$ in every representation of $\Spin_d$. We call representations for which $\varepsilon$ acts by $-1$ \emph{spin representations}. These correspond to the projective representation of $\Sym_d\mspace{1mu}$, whereas representation for which $\varepsilon$ acts by $+1$ correspond to ordinary (linear) representations of $\Sym_d\mspace{1mu}$. %The following result determines the character table of the spin symmetric group recursively (as we will explain in more detail after the theorem). 
\begin{proposition}[Schur]\label{prop:charSpin}
The irreducible spin representations $V^\lambda_\pm$ of $\Spin_d$ are para\-metrized by pairs $(\lambda,(-1)^{d-\ell(\lambda)})$ for strict partitions~$\lambda$. Moreover, the character values $\varphi^\lambda_{\pm}(x)$ are determined recursively by
\begin{enumerate}
\item[\upshape (i)] an analogue of the \emph{Murnaghan--Nakayama rule} when $\pi(x)$ has only odd cycles in its cycle type;
\item[\upshape (ii)] \[\varphi^{\lambda}_{\pm}(x) \= \pm \frac{\mathrm{i}}{\sqrt{2}}\sqrt{\prod \lambda_i}\] when $d-\ell(\lambda)$ is odd and $\pi(x)$ has cycle type~$\lambda$;
\item[\upshape (ii)] $\varphi^\lambda_{\pm}(x)=0$ in all other cases.
\end{enumerate}
\end{proposition}
The \emph{Murnaghan--Nakayama rule} for $\varphi^\lambda_{\pm}$ is \cite[Theorem~10.1]{HofHum}. There exists an amusing way to describe the Murnaghan--Nakayama rule uniformly for both $\Sym_d$ and $\Spin_d$, for which we refer to the survey by Morris \cite{Mor}. Note that in case $d-\ell(\lambda)$ is odd, both cases~{\upshape(i)} and~{\upshape(ii)} contribute one-half to the character inner product:
\begin{corollary}\label{cor:halfspin} For all $\lambda\in\SP(d)$ with $d-\ell(\lambda)$ odd
\[ \frac{1}{|\Spin_d|}\sum\nolimits_{\text{\upshape(i)}} |\varphi^{\lambda}_{\pm}(x)|^2 \= \frac{1}{2} \= \frac{1}{|\Spin_d|}\sum\nolimits_{\text{\upshape(ii)}} |\varphi^{\lambda}_{\pm}(x)|^2,\]
where the first sum is over all $x$ for which $\pi(x)$ has only odd cycles in its cycle type and the second over all $x$ for which $\pi(x)$ has cycle type~$\lambda$. 
\end{corollary}
\begin{proof}
This follows from the fact that the character inner product, which is the sum of the left and right side, is one, and that the conjugacy class of elements for which $\pi(x)$ is of type $\lambda$ is of size~$2d!/\prod \lambda_i\mspace{1mu}$.
\end{proof}
The sign of a permutation in the symmetric group determines a $\ZZ/2\ZZ$-grading on $\Sym_d$, which lifts to a $\ZZ/2\ZZ$-grading on $\Spin_d\mspace{1mu}$. Explicitly, $\deg(\varepsilon)=0$ and $\deg(t_i)=1$. The elements of degree $0$ in $\Spin_d$ form the group $\ASpin_d$, which is a central extension of the alternating group. Given a partition $\lambda$, we write $\epsilon(\lambda)$ for the parity of $d-\ell(\lambda)$. 

\begin{proposition}[Schur]\label{prop:charASpin}
The irreducible spin representations of $\ASpin_d$ are para\-metrized by pairs $(\lambda,(\pm 1)^{d-\ell(\lambda)+1})$ for strict partitions~$\lambda$. More precisely, if $d-\ell(\lambda)$ is odd, then $V^\lambda_+$ and $V^\lambda_-$ are isomorphic irreducible representations of $\ASpin_d$. If $d-\ell(\lambda)$ is even, the representation $V^\lambda_+$ splits as a sum of two irreducible representations of~$\ASpin_d$. The corresponding characters~$\alpha^{\lambda}_{\pm}$ satisfy
\begin{enumerate}
\item[\upshape (i)] $\alpha^\lambda_{\pm}(x) = 2^{\epsilon(\rho)-1}\varphi^\lambda_+(x)$ when $\pi(x)$ is of cycle type $\rho\in\OP$ and $\rho\neq \lambda$;
\item[\upshape (ii)]\[\alpha^{\lambda}_{\pm}(x) \= \frac{1}{2}\varphi^\lambda_+(x)\pm \frac{\mathrm{i}}{2}\sqrt{\prod \lambda_i}\] when $d-\ell(\lambda)$ is even and $\pi(x)$ has cycle type~$\lambda$;
\item[\upshape (ii)] $\alpha_\lambda^{\pm}(x) = \varphi^\lambda_{+}(x)=0$ in all other cases.
\end{enumerate}
\end{proposition}
Note that the values of $\varphi^\lambda_+(x)$ determined by the Murnaghan--Nakayama rule are real; hence, $\varphi^\lambda_+(x)$ is real if $x$ is as in case~(i) and purely imaginary if $x$ is as in case~(ii) in \cref{prop:charSpin}. The analogue of \cref{cor:halfspin} is therefore the following result.
\begin{corollary} For all $\lambda\in\SP(d)$ with $d-\ell(\lambda)$ even
\[ \frac{1}{|\ASpin_d|}\sum_{x\in \ASpin_d} (\Re\alpha^{\lambda}_{\pm}(x))^2 \= \frac{1}{2} \= \frac{1}{|\ASpin_d|}\sum_{x\in \ASpin_d} (\Im\alpha^{\lambda}_{\pm}(x))^2.\]
\end{corollary}

We now return to the spin symmetric group. Spin representations are representations of the \emph{twisted group algebra}
\[\mathcal{T}_d	:= \CC[\Spin_d]/(\varepsilon+1).\]
Note that $\mathcal{T}_d$ inherits the grading, and hence is a superalgebra. The irreducible supermodules of $\mathcal{T}_d$ (i.e., irreducible $\ZZ/2\ZZ$-graded modules) can easily be determined in terms of the irreducible modules of $\Spin_d$. 
\begin{proposition}
The irreducible supermodules of $\mathcal{T}_d$ are given by 
\[ V^\lambda = \begin{cases} 
V^\lambda_+ & d-\ell(\lambda) \text{ is even} \\
V^\lambda_+ \oplus V^\lambda_- & d-\ell(\lambda) \text{ is odd.}
\end{cases}
\]
\end{proposition}
If we were only to know the supermodules of $\mathcal{T}_d\mspace{1mu}$, generalities on superalgebras (see, e.g., \cite[Chapter~12]{Kle}) would allow us to obtain the irreducible supermodules of $\Spin_d$ and $\ASpin_d$. By doing so, we would almost be able to recover \cref{prop:charSpin} and \cref{prop:charASpin}.
There is, however, an important subtlety: by doing so we lose information about the characters $\varphi^\lambda_{\pm}$ and $\alpha^\lambda_{\pm}$ at $x$ for which $\pi(x)$ has cycle type~$\lambda$, and it would require more specific information about the supermodules to recover part~(ii) in these results. Soon, we will come across the same subtlety for a different supermodule. 

Instead of working with $\mathcal{T}_d\mspace{1mu}$, often it is more convenient to work with the Sergeev superalgebra. Both algebras capture the same information. Before introducing this algebra, we introduce two more groups. Let
\[\Cliff_d \defis \{\xi_1,\ldots,\xi_d,\varepsilon \mid \varepsilon^2=1,\ \xi_i^2=\varepsilon,\ \varepsilon \xi_i=\xi_i \varepsilon,\ \xi_i\xi_j=\varepsilon\xi_j\xi_i \text{ for all } i\neq j\}\]
be the \emph{Clifford group}, which is a central extension of $(\ZZ/2\ZZ)^d$. Following Sergeev \cite{Ser}, we define the \emph{Sergeev group} $\Sergeev_d$ as the semidirect product
\[ \Sergeev_d \defis \Sym_d \ltimes \Cliff_d,\]
where $\Sym_d$ acts on $\Cliff_d$ by permuting the $\xi_i$.  
(Some authors define the Sergeev group slightly differently by setting $\xi_i^2=1$ instead of $\xi_i^2=\varepsilon$. The representation theory of these two groups is the same; note, however, that the two Sergeev groups are non-isomorphic.)

Before we study the representation theory of $\Sergeev_d\mspace{1mu}$, we describe the conjugacy classes~$C$ such that $C\cap\varepsilon C =\emptyset$. Namely, if the latter condition is not satisfied, every spin representation is trivial on $C$. 
\begin{lemma}[{\cite[Lemma~5]{Ser}}] Let $g=(\sigma, \xi) \in \Sergeev_d$ %with $\sigma \in \Sym_d$ and $\xi\in \Cliff_d\mspace{1mu}$. 
and write $\xi=\varepsilon^{a_0}\xi_1^{a_1}\cdots \xi_n^{a_n}$ for $a_i\in \{0,1\}$. Then, $g$ is not conjugate to $\varepsilon g$ if and only if 
\begin{enumerate}
\item[{\upshape (1)}] $\deg(\xi)=0$, the cycle type of $\sigma$ is in $\OP$, and $\sum_{j\in \tau} a_j$ is even for all disjoint cycles~$\tau$ of $\sigma$;
\item[{\upshape (2)}] $\deg(\xi)=1$, the cycle type of $\sigma$ is $\lambda\in \SP$ with $\ell(\lambda)$ odd and $\sum_{j\in \tau} a_j$ is odd for all disjoint cycles~$\tau$ of $\sigma$.
\end{enumerate}
\end{lemma} 
The first case corresponds to the conjugacy class of a pure permutation $(\sigma,1)\in \Sergeev_d$. Note that this conjugacy class only depends on the cycle type $\rho\in \OP$ of $\sigma$; we denote this conjugacy class by~$C_\rho\mspace{1mu}$. In the second case, we write $C_{\lambda,1}$ for the corresponding conjugacy class. 

On $\Cliff_d$  we introduce the $\ZZ/2\ZZ$-grading by setting $\deg \xi_i=1$. This grading extends to $\Sergeev_d$ by additionally letting $\deg \sigma = \deg \varepsilon =0$ for $\sigma \in \Sym_d\mspace{1mu}$. Define the corresponding \emph{twisted algebras} by
\[ \mathcal{C}_d \defis \CC[\Cliff_d]/(\varepsilon+1), \qquad \mathcal{X}_d \defis \CC[\Sergeev_d]/(\varepsilon+1), \]
which both are superalgebras. The following results determine the corresponding irreducible supermodules. In these results, it is important to recall that the tensor product of two superalgebras~$\mathcal{A}$ and $\mathcal{B}$ is not the same as the tensor product of two algebras. Instead, the multiplication is defined by
\[ (a\otimes b)(a'\otimes b') = (-1)^{\deg(b)\deg(a')}(aa')\otimes (bb') \qquad (a,a'\in \mathcal{A}, b,b'\in \mathcal{B}).\]
\begin{proposition}
The Clifford superalgebra~$\mathcal{C}_d$ is irreducible of dimension~$2^{\lfloor \frac{d+1}{2}\rfloor}$. 
\end{proposition}
In particular, the character $\zeta$ corresponding to $\mathcal{C}_d$ satisfies $\zeta(1)=2^{\lfloor \frac{d+1}{2}\rfloor}$ and $\zeta(x)=0$ if $x\neq 1$. 
\begin{proposition}
The map $\vartheta_d:\mathcal{T}_d \otimes \mathcal{C}_d \to \mathcal{X}_d$ given by
\begin{align*}
t_j\otimes 1 &\mapsto \frac{1}{\sqrt{2}}(\xi_j-\xi_{j+1})(j,j+1) \\
 1\otimes \xi_j&\mapsto \xi_j
\end{align*}
is an isomorphism of superalgebras. 
\end{proposition}
We write $\delta(\lambda)$ and $\epsilon(\lambda)$ for the parity of $\ell(\lambda)$ and $d-\ell(\lambda)$ respectively.  
\begin{corollary}\label{cor:Xd}
The irreducible supermodules of $\mathcal{X}_d$ are given by $V_\lambda\otimes \mathcal{C}_d$, where~$\lambda$ goes over all strict partitions. 
The corresponding character $\theta^\lambda$ is given by 
\[ \theta^\lambda(x) \= \begin{cases} 2^{\frac{\ell(\rho)+\delta(\lambda)-\epsilon(\lambda)}{2}}\varphi^\lambda(\rho) & x \in C_\rho \text{ for some } \rho \in \OP\\
0 & x\not \in C_\rho \cup \varepsilon C_\rho \text{ for some } \rho \in \OP,
\end{cases}
\]
where $\varphi^\lambda(\rho)$ is the character of $V_\lambda$ evaluated at a permutation of type $\rho\in\OP$.
\end{corollary}

The structure of \emph{super}modules of $\mathcal{X}_d$ does not directly imply the character table of the Sergeev group, nor does seem to be contained in the literature. However, the irreducible representations of $\Sergeev_d$ were constructed by Maxim Nazarov in \cite[Section~1]{Naz97}. We deduce the following proposition from his result.
\begin{proposition}\label{prop:charSergeev}
The irreducible spin representations of $\Sergeev_d$ are parametrized by pairs $(\lambda,(\pm 1)^{\ell(\lambda)})$ for strict partitions~$\lambda$. Moreover, the character values $\chi^\lambda_{\pm}(x)$ are determined recursively by
\begin{enumerate}
\item[\upshape (i)] $\chi^\lambda_{\pm}(x)=2^{-\delta(\lambda)}\theta^\lambda(\rho)$ if $x\in C_\rho$.
\item[\upshape (ii)]\[\chi^{\lambda}_{\pm}(x) \= \pm 2^{(\ell(\lambda)-1)/2}
 \mathrm{i}\sqrt{\prod \lambda_i} 
\] when $\ell(\lambda)$ is odd and $x\in C_{\lambda,1}.$
\item[\upshape (ii)] $\chi^\lambda_{\pm}(x)=0$ in all other cases.
\end{enumerate}
\end{proposition}
\begin{proof}
First of all, observe that by the previous corollary the result holds if~$\ell(\lambda)$ is even. Namely, in that case, the irreducible representation corresponding to $\lambda$ equals the \emph{super}representation corresponding to $\lambda$. We now follow Nazarov's construction in the case $\ell(\lambda)$ is odd.

First, assume $d$ is even. Write $\tau^\lambda_{\pm}$ for the representations corresponding to $V^\lambda_{\pm}$. Then the irreducible representation of $\Sergeev_d$ corresponding the pair $(\lambda,\pm)$ in the space $\mathcal{C}_d \otimes V^\lambda_{\pm}$ is given by $(j=1,\ldots,d-1)$
\begin{align*}
 (j,j+1) &\mapsto \ii \xi_0\frac{\xi_j-\xi_{j+1}}{\sqrt{2}} \otimes \tau^\lambda_{\pm}(t_j), \qquad
 \xi_j \mapsto \xi_j\otimes \mathrm{id},\qquad 
 \varepsilon  \mapsto -1.
\end{align*}
Let $x\in \Sergeev_d\mspace{1mu}$. Then, as the trace of the Clifford algebra only takes values on the identity, we find
\[ \chi^\lambda_{\pm}(x) = C(x)\frac{2^{\frac{d}{2}}}{2^{\frac{d-\ell(\lambda)}{2}}}\varphi^{\lambda}_{\pm}(\rho)\]
where $C(x)$ denotes the multiplicity of the identity in $\xi_0(\xi_{1}-\xi_{2})\cdots \xi_0(\xi_{\lambda_{j_1}-1}-\xi_{\lambda_{j_r}})\xi$ if we write $x=(j_1,j_1+1)\cdots(j_r,j_r+1)\xi$ and where $\rho$ is the cycle type of the projection of $x$ to the symmetric group. Note that this multiplicity takes values in $\{-1,0,1\}$. In case $\rho\in\OP$, we obtain $\chi^\lambda_{\pm}(x)=2^{\frac{\ell(\lambda)}{2}}\varphi^\lambda_{\pm}(\rho)=2^{\frac{\ell(\lambda)-1}{2}}\theta^\lambda(\rho)$, which we could also have deduced directly from \cref{cor:Xd}. More interestingly, if $x\in C_{\lambda,1}$, we obtain
\[\chi^\lambda_{\pm}(x) = \pm 2^{(\ell(\lambda)-1)/2}\mathrm{i}\sqrt{\prod \lambda_i}.\]
That $\chi^\lambda_\pm(x)=0$ in all other cases, now easily follows from the orthogonality relations, or by a similar computation as above. 

Next, suppose $d$ is odd. Note that $\mathcal{C}_d$, considered as an ordinary module, rather than a super module, splits as a sum $\mathcal{C}_d \simeq \mathcal{C}_d^+\oplus \mathcal{C}_d^-$ of irreducible (non-super) modules $\mathcal{C}_d^{\pm}$ of dimension~$2^{\frac{d-1}{2}}$. Let $I=\left(\begin{smallmatrix} 0 & -\ii \\ \ii & 0\end{smallmatrix}\right)$ and $J=\left(\begin{smallmatrix} 1 & 0 \\ 0 & -1\end{smallmatrix}\right)$. Nazarov now first constructs a reducible representation of $\Sergeev_d$ corresponding to $\lambda$ in the space $\CC^2\otimes \mathcal{C}_d^+ \otimes V^\lambda$ by
\[(j,j+1)\mapsto I\otimes \frac{\xi_j-\xi_{j+1}}{\sqrt{2}} \otimes \tau_\lambda(t_j), \qquad \xi_j \mapsto J\otimes \xi_k\otimes \mathrm{id}.\]
Write $\omega_\lambda$ for the endomorphism of $V_\lambda$ defined by $v\mapsto (-1)^{\deg(v)} v$. Then, the $\pm 1$-eigenspaces with respect to the involution $J\otimes \mathrm{id}\otimes \omega_\lambda$ form two irreducible representations corresponding to $(\lambda,\pm)$. In particular, in this case, we obtain
\[\chi^\lambda_{\pm}(x) \=  D(x)  \frac{2^{\frac{d-1}{2}}}{2^{\frac{d-\ell(\rho)}{2}}}\cdot
\begin{cases}
\alpha_{+}^\lambda(\rho)+\alpha_{-}^\lambda(\rho) & \deg x=0 \\ 
\alpha_{+}^\lambda(\rho)-\alpha_{-}^\lambda(\rho) & \deg x=1
\end{cases}
\]
with $D(x)$ the multiplicity of the identity in $(\xi_{j_1}-\xi_{j_1+1})\cdots (\xi_{j_2}-\xi_{\lambda_{j_2+1}})\xi$ if we write $x=(j_1,j_1+1)\cdots(j_r,j_r+1)$, and where the product $(j_1,j_1+1)\cdots(j_r,j_r+1)$ is of cycle type $\rho$ (i.e., $x\in C_\rho\xi$). In other words, if $x\in C_\rho$, we obtain
$\chi^\lambda_{\pm}(x) = 2^{(\ell(\lambda)+1)/2}\alpha^\lambda_{\pm}(\rho) = 2^{(\ell(\lambda)+1)/2}\alpha^\lambda_{\pm}(\rho) = \frac{1}{2}\theta^\lambda(\rho)$.  Moreover, if $x\in C_{\lambda,1}$, we obtain
\begin{align*}
\chi^\lambda_{\pm}(x) = \pm 2^{(\ell(\lambda)-1)/2}\ii\sqrt{\prod\lambda_i}.&\qedhere
\end{align*}

\end{proof}
Similar to the spin symmetric group, we conclude that both cases contribute one-half to the inner product of characters.
\begin{corollary}\label{cor:1/2} For all $\lambda\in \SP$ with $\ell(\lambda)$ odd
\[
\frac{1}{|\Sergeev_d|}\sum_{\rho\in\OP}\sum_{x\in C_\rho} |\chi^{\lambda}_{\pm}(x)|^2 \= \frac{1}{2} \= \frac{1}{|\Sergeev_d|}\sum_{x\in C_{\lambda,1}} |\chi^{\lambda}_{\pm}(x)|^2 \]
\end{corollary}

Write $\Sergeev_d^0$ for the subgroup of $\Sergeev_d$ consisting of elements of even degree. Similar to \cref{prop:charSergeev}, we obtain the character table. 
\begin{proposition}\label{prop:charSergeev0}
The irreducible spin representations of $\Sergeev_d^0$ are parametrized by pairs $(\lambda,(\pm 1)^{\ell(\lambda)+1})$ for strict partitions~$\lambda$. Moreover, the character values $\psi^\lambda_{\pm}(x)$ are determined recursively by
\begin{enumerate}
\item[\upshape (i)] $\psi^\lambda_{\pm}(x)=\frac{1}{2}\theta^\lambda(\rho)$ if $x\in C_\rho$ and $\rho\neq \lambda$.
\item[\upshape (ii)]\[\psi^{\lambda}_{\pm}(x) \= \frac{1}{2}\theta^\lambda(\rho)\pm \frac{1}{2} 2^{\ell(\lambda)/2} \mathrm{i}\sqrt{\prod \lambda_i}
\] when $\ell(\lambda)$ is even and $x\in C_{\lambda}.$
\item[\upshape (ii)] $\psi^\lambda_{\pm}(x)=0$ in all other cases.
\end{enumerate}
\end{proposition}
\begin{corollary}\label{cor:half0}  For all $\lambda\in \SP(d)$ with $\ell(\lambda)$ even
\[ \frac{1}{|\Sergeev_d^0|}\sum_{x\in \Sergeev_d^0} (\Re\psi^{\lambda}_{\pm}(x))^2 \= \frac{1}{2} \= \frac{1}{|\Sergeev_d^0|}\sum_{x\in \Sergeev_d^0} (\Im\psi^{\lambda}_{\pm}(x))^2.\]
\end{corollary}

Finally, we end our discussion by introducing the central characters associated to the Sergeev group. 
Given $\rho\in \OP(d)$, the class sum $\overline{C_\rho}$, which is the sum of all permutations in $C_\rho$, acts by Schur's lemma as a constant on a supermodule~$V_\lambda\otimes \mathcal{C}_d$. We call this constant \emph{the central character} and denote it by~$\mathbf{f}_\rho(\lambda)$. In fact,
\begin{align}\label{eq:centralchar} \mathbf{f}_\rho(\lambda) \= |C_\rho|\frac{\theta^\lambda(\rho)}{\dim\theta^\lambda} \= |C_\rho|\frac{\chi^\lambda_{\pm}(\rho)}{\dim\chi^\lambda_{\pm}} \= |C_\rho|\frac{\psi^\lambda(\rho)}{\dim\psi^\lambda},
\end{align}
where $\psi^\lambda$ is the average of $\psi^\lambda_+$ and $\psi^\lambda_-$ (or equal to $\psi^\lambda_+$ if $\ell(\lambda)$ is odd). In particular, the central characters associated to $\Sergeev_d$ and $\Sergeev_d^0$ agree. (Note that we restricted ourselves to $\rho\in\OP$ in the definition of central characters.)

We extend this notation to all $\rho\in\OP$ and $\lambda \in \SP$, even if $\rho$ and $\lambda$ are not of the same size. Write $\rho\sim \rho'$ if the partitions $\rho$ and $\rho'$ only differ in the number of parts equal to~$1$, and set
 \begin{align}\label{eq:cc}
  \mathbf{f}_\rho(\lambda) \= \begin{cases} 
 \mathbf{f}_{\rho'}(\lambda) & |\rho|<|\lambda|, \qquad  \\
 0 & |\rho|>|\lambda|,
 \end{cases}
 \end{align}
where $\rho\sim\rho'$ and $|\rho'|=|\lambda|$. Then, by \cite{Iva} we have $\mathbf{f}_\rho\in \Lambda$, where $\Lambda$ is the symmetric algebra introduced before. More concretely, if $\rho=(\ell)$, then by \cite[Theorem~6.7]{CheMoeSauZag} 
%(making a small correction by writing $p_j$ instead of $\mathbf{p}_j$)
\[\ell\,\mathbf{f}_\ell(\lambda) = \frac{-1}{2\ell} [t^{\ell+1}]\biggl(\prod_{j=1}^{\ell-1}(1-jt) \cdot \exp\Bigl(\sum_{\substack{k\geq 1 \\ k\, {\rm odd}}} \frac{2}{k}p_k(\lambda)\, t^k(1-(1-\ell t)^{-k})\Bigr)\biggr).\]
(The central characters in our work agree with those in \cite{EskOkoPan}, but differ by a factor~$\ell$ of those in \cite[Theorem~6.7]{CheMoeSauZag} and in \cite[Definition~6.3]{Iva}.)
By comparing this formula with the definition of $\mathbf{h}_\ell$, we see that $\mathbf{f}_\ell- \mathbf{h}_\ell/\ell$ is of weight less than $\ell+1$. %\cite[Theorem~6.7]{CheMoeSauZag}.

%%%%%%%%%%%%%%%%%%%%%%%
\subsection{Weighted spin Hurwitz numbers}\label{sec:spinHurwitz}
%%%%%%%%%%%%%%%%%%%%%%%
Let $\Pi=(\mu^{(1)},\ldots,\mu^{(n)})$ with $\mu^{(i)}\in \OP(d)$. A \emph{Hurwitz tuple of degree~$d$ and ramification type~$\Pi$} is an element 
\[(\alpha,\beta,\gamma_1,\ldots,\gamma_n)\in (\Sym_d)^{n+2}\]
 such that $[\alpha,\beta]\gamma_1\ldots\gamma_n = 1$ and the type of~$\gamma_i$ is equivalent to the partition~$\mu^{(i)}$ (i.e., they only differ in the amount of $1$'s).

Write $\Hur_d(\Pi)$ for the set of all Hurwitz tuples $h$ of degree~$d$ and ramification type~$\Pi$. Every such Hurwitz tuple corresponds to a (ramified) covering of the torus, which, after pulling back the flat metric on the torus, yields a differential of type~$\mu$. This induces a spin parity $s:\Hur(\Pi)\to \ZZ/2\ZZ$. By \cite[Theorem~2]{EskOkoPan} the \emph{spin Hurwitz number} of degree~$d$ and ramification profile~$\Pi$ is given by
\[\frac{1}{d!}\sum_{h\in \Hur_d(\Pi)} (-1)^{s(h)}\= 2^{\chi/2} \sum_{\lambda \in \SP(d)} (-1)^{\ell(\lambda)}\, \mathbf{f}_{\Pi}(\lambda), \]
where $\mathbf{f}_\Pi=\prod_{i} \mathbf{f}_{\mu^{(i)}}$ with $\mathbf{f}_{\mu^{(i)}}$ a central character, defined by~\eqref{eq:cc}, and $\chi$ is the Euler characteristic of the cover, i.e., 
\[ \chi \= \sum_{i} \bigl(\ell(\mu^{(i)})-|\mu^{(i)}|\bigr).\]
We generalize this result, by finding the following expression for a weighted count. Given a Hurwitz tuple $h=(\alpha,\beta,\gamma_1,\ldots,\gamma_n)$ and $f\in\Lambda$, write $f(\alpha)$ to denote the value of $f$ applied to the partition corresponding to the conjugacy class of~$\alpha$.
\begin{proposition} For any $f\in \Lambda$ we have
\[\frac{1}{d!}\sum_{h\in \Hur_d(\Pi)} (-1)^{s(h)} f(\alpha) \= 2^{\chi/2} \sum_{\lambda \in \SP(d)} (-1)^{\ell(\lambda)}\, \mathbf{f}_{\Pi}(\lambda)\,  f(\lambda). \]
\end{proposition}
\begin{proof}
%For any finite group~$G$, denote by $\mathcal{C}$ the set of all conjugacy classes of $G$. 
Given conjugacy classes $C_1,\ldots,C_n$ in a finite group $G$, we define \emph{a Hurwitz tuple for~$G$ with ramification type $\mathcal{C}=(C_1,\ldots,C_n)$} to be an element
\[ (\alpha,\beta,\gamma_1,\ldots,\gamma_n) \in G^{n+2}\]
such that $[\alpha,\beta]\gamma_1\cdots\gamma_n =1$ and $\gamma_i\in C_i$ for all $i$. Denote by $\Hur_G(\mathcal{C})$ the set of all Hurwitz tuples for~$G$ with ramification type $\mathcal{C}$. The sum of all elements of a conjugacy class $C$ in the group algebra of $G$ acts by Schur's lemma by a constant; this constant is the central character $f_C$, which we consider as a function $f_C:G^{\wedge}\to\CC$, where $G^{\wedge}$ denotes the set of irreducible representations of $G$. We write $f_{\mathcal{C}}=\prod_{C\in \mathcal{C}} f_C$. Given an irreducible representation $\pi$ of $G$ and a class function $f:G\to \CC$, we let
\[ M_G(f,\pi) \defis \frac{1}{|G|}\sum_{g\in G} \chi_{\pi}(g)\,\overline{\chi_{\pi}(g)}\,f(g)
\= \frac{1}{|G|}\sum_{[g]} |[g]|\, \chi_{\pi}(g)\,\overline{\chi_{\pi}(g)}\,f(g), \]
where the last sum is over all conjugacy classes $[g]$ of $G$, and with $\chi_\pi$ the character of the representation~$\pi$ and $|[g]|$ the size of the conjugacy class of $g$ in $G$.

Mutatis mutandis, the proof of \cite[Theorem~A.1.10]{LanZvo} implies that for all class functions $f:G\to \CC$, we have that
\[ H_{f,\mathcal{C}}(G)\defis \frac{1}{|G|}\sum_{h\in \Hur_G(\mathcal{C})}\!\! f(\alpha) \= \sum_{\pi\in G^{\wedge}} f_{\mathcal{C}}(\pi) \, M_G(f,\pi),
\]
where $h=(\alpha,\beta,\gamma_1,\ldots,\gamma_n)$. % and the sum over $\pi$ ranges over all irreducible representations of $G$.
In the non-spin setting this result for $G=\Sym_d$ suffices to prove the non-spin variant of this proposition (see \cite[Proposition~6.3]{CheMoeZag}). Here, it does not suffice to let $G=\Sergeev_d$, as we have to take care of the sign of the Hurwitz tuple. Hence, we complete the proof along the same lines as  \cite[\S\S3.1.6--3.1.10]{EskOkoPan}.

Given a group homomorphism $\phi: G'\to G, h\in \Hur_{G}(\mathcal{C})$ and a conjugacy class $\mathcal{C}'$ of $G'$ such that $\phi(\mathcal{C}')=\mathcal{C}$, we write
$\Hur_{G'}(h,\phi)=\Hur_{G'}(h)$ for the set of Hurwitz tuples $h'\in \Hur_{G'}(\mathcal{C}')$ such that $\phi(h') = h$. For all $d\geq 1$, write $\SSym_d$ for the group of signed permutations, obtained by setting $\varepsilon=1$ in $\Sergeev_d\mspace{1mu}$. Moreover, write $\SSym_d^0$ and $\Sergeev_d^0$ for the subgroups of even elements with respect to the $\ZZ/2\ZZ$-grading on~$\Sergeev_d$. Note that all these four groups $\SSym_d,  \SSym_d^0, \Sergeev_d^0$ and $\Sergeev_d^0$ admit a natural projection homomorphism to~$\Sym_d$. 

Write $\Hur_{d}(\mathcal{C})$ for $\Hur_{\Sym_d}(\mathcal{C})$ and let $h\in\Hur_{d}(\mathcal{C})$ be given. Now, in terms of Hurwitz tuples, \cite[Proposition~5]{EskOkoPan} reads
\[ \frac{(-1)^{s(h)}}{d!} = 2^{\chi/2} \biggl(\frac{|\Hur_{\SSym_d}(h)|}{|\SSym_d|}- \frac{|\Hur_{\SSym^0_d}(h)|}{|\SSym^0_d|}- \frac{|\Hur_{\Sergeev_d}(h)|}{|\Sergeev_d|}+ \frac{|\Hur_{\Sergeev^0_d}(h)|}{|\Sergeev^0_d|}\biggr).\]
%Note that in all these four sets the element~$\alpha$ is lifted to an element in the same conjugacy class. 
Hence, summing over all $h\in \Hur_d(\mathcal{C})$ and multiplying by $f(\alpha)$, we obtain
\[ \sum_{h\in \Hur_d(\mathcal{C})}\!\!\! \frac{(-1)^{s(h)} }{d!} f(\alpha) \= 2^{\chi/2} \bigl(
H_{\! f,\mathcal{C}}({\SSym_d})-H_{\! f,\mathcal{C}}({\SSym^0_d})-H_{\! f,\mathcal{C}}({\Sergeev_d})+H_{\! f,\mathcal{C}}({\Sergeev^0_d})\bigr),
\]
where in $H_{\! f,\mathcal{C}}(G)$ the conjugacy class~$\mathcal{C}$ should be interpreted as the conjugacy class of pure permutations in $G$. 

%We now proceed as in \cite[\S\S3.1.6--3.1.10]{EskOkoPan}. 
Observe that the irreducible representations $\pi$ of $\Sergeev_d$ such that $\pi(\varepsilon)=1$ correspond to the representations of~$\SSym_d$. In particular, $M_{\SSym_d}(f,\pi)=M_{\Sergeev_d}(f,\pi)$ for such~$\pi$. Hence,
\[ H_{f,\mathcal{C}}({\Sergeev_d})-H_{f,\mathcal{C}}({\SSym_d})\=\sum_{\pi\in (\Sergeev_d)^{\wedge}_{-}} f_\mathcal{C}(\pi) \, M_{\Sergeev_d}(f,\pi),\]
where $(\Sergeev_d)^{\wedge}_{-}$ denotes the set of irreducible spin representations of~$\Sergeev_d$. Similarly, the results holds after replacing $\Sergeev_d$ and $\SSym_d$ with $\Sergeev^0_d$ and $\SSym^0_d$, so that
\[ \sum_{h\in \Hur_d(\mathcal{C})}\!\!\! \frac{(-1)^{s(h)} }{d!} f(\alpha) \= \!\sum_{\pi\in (\Sergeev^0_d)^{\wedge}_{-}}\! \! f_\mathcal{C}(\pi) \, M_{\Sergeev^0_d}(f,\pi) \meno \!\sum_{\pi\in (\Sergeev_d)^{\wedge}_{-}}\! \! f_\mathcal{C}(\pi) \, M_{\Sergeev_d}(f,\pi) .\]

At this point, we can no longer follow \cite{EskOkoPan} closely. Instead, we compute $M_{\Sergeev_d}(f,\pi)$ using \cref{prop:charSergeev}. First assume that $\pi=\pi_\lambda$ is associated to $\lambda\in \SP$ and $\ell(\lambda)$ is odd. Then, $\pi_{\lambda}$ only takes values on the conjugacy classes $C_\rho$ and $\varepsilon C_\rho$. %Note $\chi_{\pi}(\varepsilon\rho)=-\chi_{\pi}(\rho)$, as $\varepsilon$ acts by $-1$. 
Both conjugacy classes consist of $2^{d-\ell(\rho)}\frac{d!}{z_\rho}$ elements. Hence, we obtain
\begin{align*}
 M_{\Sergeev_d}(f,\pi_\lambda) 
 &\= \frac{1}{2^{d+1}d!}\sum_{\rho \in \OP(d)} 2\cdot 2^{d-\ell(\rho)}\frac{d!}{z_\rho} |\chi_\lambda(\rho)|^2 f(\rho) \\
  &\= \sum_{\rho \in \OP(d)} \frac{2^{-\ell(\rho)}}{z_\rho} |\chi_\lambda(\rho)|^2 f(\rho) =:\Moller(f)(\lambda),
% &\=\sum_{\rho \in \OP(d)} \frac{2^{\ell(\rho)}}{z_\rho} |\pmb{\chi}_\lambda(\rho)|^2 f(\rho).
 \end{align*}
where $\Moller(f)$ can be thought of as the spin analogue of the M\"oller transform in \cite[Corollary 13.2]{CheMoeZag}. 
Next, let $\pi_{\lambda}^\pm$ be associated to $\lambda \in \SP$ with $\ell(\lambda)$ odd. 
By \cref{cor:1/2}, we find that
\begin{align*}
 M_{\Sergeev_d}(f,\pi_\lambda^{\pm}) 
  &\= \frac{1}{2}f(\lambda)+\sum_{\rho \in \OP(d)} \frac{2^{-\ell(\rho)}}{z_\rho} |{\chi}_\lambda^{\pm}(\rho)|^2 f(\rho)
% &\= \frac{1}{2}f(\lambda)+\sum_{\rho \in \OP(d)} \frac{2^{\ell(\rho)}}{z_\rho} \frac{|\pmb{\chi}_\lambda(\rho)|^2}{4} f(\rho)\\
%&\= \frac{1}{2}f(\lambda)+\frac{1}{2}\sum_{\rho \in \OP(d)} \frac{1}{z_\rho} |X_\lambda(\rho)|^2 f(\rho) \\
   =: \frac{1}{2}f(\lambda)+\frac{1}{2}\Moller(f)(\lambda),
 \end{align*}
where we now also defined the M\"oller transform if $\ell(\lambda)$ is odd. 

Similarly, we compute $M_{\Sergeev^0_d}(f,\pi)$ for all $\pi \in (\Sergeev^0_d)^\wedge$ using \cref{prop:charSergeev0}. First, suppose $\pi=\pi_\lambda$ and $\ell(\lambda)$ is odd. As $\psi^\lambda_+=\chi^\lambda_+$, we find
\begin{align*}
 M_{\Sergeev^0_d}(f,\pi_\lambda) 
 &\= \frac{1}{2^{d}d!}\sum_{\rho \in \OP(d)} 2\cdot 2^{d-\ell(\rho)}\frac{d!}{z_\rho} |\chi^\lambda_{+}(\rho)|^2 f(\rho)
% &\=2\sum_{\rho \in \OP(d)} \frac{2^{-\ell(\rho)}}{z_\rho} \frac{|\pmb{\chi}_\lambda(\rho)|^2}{4} f(\rho) 
\= \Moller(f)(\lambda).
 \end{align*}
Finally, let $\pi_{\lambda}^{\pm} \in (\Sergeev^0_d)^\wedge$ with $\ell(\lambda)$ even. By \cref{cor:half0}, we obtain
 \begin{align*}
 M_{\Sergeev^0_d}(f,\pi_\lambda^{\pm}) 
  &\= \frac{1}{2}f(\lambda)+2\sum_{\rho \in \OP(d)} \frac{2^{-\ell(\rho)}}{z_\rho} \Bigl|\frac{1}{2}{\chi}_\lambda(\rho)\Bigr|^2 f(\rho)\\
% &\= \frac{1}{2}f(\lambda)+\sum_{\rho \in \OP(d)} \frac{2^{\ell(\rho)+1}}{z_\rho} \frac{|\pmb{\chi}_\lambda(\rho)|^2}{4} f(\rho)\\
%&\= \frac{1}{2}f(\lambda)+\frac{1}{2}\sum_{\rho \in \OP(d)} \frac{1}{z_\rho} |X_\lambda(\rho)|^2 f(\rho) \\
   &\= \frac{1}{2}f(\lambda)+\frac{1}{2}\Moller(f)(\lambda).
 \end{align*}
 
Recall that the central characters corresponding to all representations of~$\Sergeev_d$ and~$\Sergeev^0_d$ %associated to some $\lambda\in \SP$ 
agree on odd partitions. Hence, we conclude
\begin{align*}  \sum_{h\in \Hur_d(\mathcal{C})}\!\! &\frac{(-1)^{s(h)} }{d!} f(\alpha) \\
\= 
& 2^{\chi/2} \sum_{\lambda \in \SP(d)} \mathbf{f}_{\Pi}(\lambda)\, 
(-1)^{\ell(\lambda)} \bigl(M(f)(\lambda)-2\bigl(\tfrac{1}{2}f(\lambda)+\tfrac{1}{2}M(f)(\lambda)\bigr)\bigr) \\
\= & 2^{\chi/2} \sum_{\lambda \in \SP(d)} (-1)^{\ell(\lambda)}\, \mathbf{f}_{\Pi}(\lambda)\, f(\lambda).
 &\qedhere
\end{align*}
\end{proof}

Define the \emph{(combinatorial) $\ell$-weighted spin Siegel--Veech constant $c_{\ell}(d,\Pi)$} to be %the sum of the weights over all Hurwitz classes for
\[ c_{\ell}(d,\Pi) \defis \frac{1}{n\,d!}\sum_{h\in \Hur_d(\Pi)} (-1)^{s(h)}\,p_\ell(h),\]
where $p_\ell(h)$ is the \emph{$\ell$-th Siegel--Veech weight} of the Hurwitz tuple $h=(\alpha,\beta,\gamma_1,\ldots,\gamma_n)$, which by \cite[Lemma~6.1]{CheMoeZag} is given by
$p_\ell(h) \= n \, p_\ell(\alpha).$
%and where $[\alpha]$ denotes the cycle type of $\alpha$. 

\begin{corollary}\label{cor:weightedc0} For all odd $\ell$ 
\[ c_{\ell}(d,\Pi) \= 2^{\chi/2} \sum_{\lambda \in \SP(d)} (-1)^{\ell(\lambda)} \mathbf{f}_\Pi(\lambda) \, p_\ell(\lambda).
%(T_{\ell}^\odd(\lambda)-p_{\ell-1}^\odd(\lambda)).
\]
\end{corollary}

%%%%%%%%%%%%%%%%%%%%%%%
\subsection{Recursion relation for $c_0^\spin(\mu)$}\label{sec:recursion}
%%%%%%%%%%%%%%%%%%%%%%%

%For all $\ell\geq 1$ odd, we define:
%$$
%\mathbf{h}_\ell = \frac{-1}{\ell} \left[u^{\ell+1}\right]  P(u)^\ell, \text{where } P(u)=\sum_{s\geq 0} u^{2s+1} \mathbf{p}_{2s+1}.
%$$
The symmetric algebra~$\Lambda$ is canonically identified with $\QQ[\mathbf{h}_1,\mathbf{h}_3,\ldots]$ (see \eqref{def:h} for the definition of the $\mathbf{h}_i$). As in \cite[Equation~(56)]{CheMoeSauZag}, for all non-empty sets $I\subset \NN$ of cardinality~$n$, we define a function $\Hfunc_{I}\in R[\![(z_{i})_{i\in I}]\!]$ as follows:
\begin{eqnarray*}
\Hfunc_{\{i\}} &:=& z_i^{-1} + \sum_{s\geq 0} \mathbf{h}_{2s+1} z_i^{2s+1},\\
\Hfunc_{\{i,j\}}&:=& \frac{z_i\cH'(z_i)-z_j\Hfunc'(z_j)}{\Hfunc(z_j)-\Hfunc(z_i)}-1,\text{ and}\\
\Hfunc_I &:=& \frac{1}{(n-1)} \sum_{\substack{ k\geq 0 \\  \underline{\ell}=(\ell_1\ldots, \ell_k)}} \sum_{I=\{r,s\} \sqcup I_1\sqcup\ldots\sqcup I_k } \frac{1}{k!} \Hfunc^{\underline{\ell}}_{\{r,s\}}\prod_{i=1}^k \Hfunc_{I_i}^{[\ell_i]} \qquad (n\geq 3).
\end{eqnarray*}
In the last line, the first sum on the right-hand side is over all vectors of odd positive integers of length $k$, while the second sum is over partitions of $I$ into $k+1$ non-empty sets, and we let
$$
\Hfunc_{\{i,j\}}^{\underline{\ell}} \defis \frac{\partial^k}{\partial_{\ell_{1}}\ldots \partial_{\ell_k}} \Hfunc_{\{i,j\}},\quad \text{ and } \quad \Hfunc_{I}^{[\ell]}\defis [z_i^\ell] \Hfunc_{I\cup\{i\}}.
$$
Then for any element $\mathbf{h}\in R$ we denote by $\mathbf{h}|_{\mathbf{h}_\ell\mapsto \boldsymbol{\alpha}_\ell}\in \QQ$ the image of $\mathbf{h}$ under the unique ring morphism $\Lambda\to \mathbb{Q}$ mapping $\mathbf{h}_\ell$ to the rational number $\boldsymbol{\alpha}_\ell$ defined by:
$$
\boldsymbol{\alpha}_\ell \defis \frac{-1}{2\ell}[u^\ell]\, \mathbf{P}(u)^\ell, \quad \text{ where } \quad \mathbf{P}(u) \defis {\rm exp}\Bigl(\sum\nolimits_{s\geq 1} \zeta(-s)\,u^{s+1}\Bigr).
$$
\begin{notation}\label{not}
If $X$ is a connected component of a stratum $\cH(\mu)$, the {\em normalized Masur--Veech volume} of $X$ is defined as
$$
\v(X) \defis \frac{(|\mu|-1)! \prod_{i=1}^n m_i}{2(2\pi\ii)^{2g}} \Vol(X).
$$
We denote $\v(\mu)=\v(\H(\mu)),$ and $\v^\spin(\mu)=\v(\H(\mu)^\even)-\v(\H(\mu)^\odd)$.
\end{notation}
Recall that the leading term $\langle \cdots \rangle_L$ of the growth polynomial is defined by~\eqref{eq:L}, the elements $\mathbf{h}_\ell\in\Lambda$ are defined by~\eqref{def:h} and recall that $\partial_2=\pdv{}{\mathbf{p}_1}$. 
\begin{theorem}\label{th:withqmf} We have
\begin{eqnarray*}   c_0^\spin(\mu) &=& \frac{3\cdot 2^{\chi/2}}{ 2 \pi^2 \cdot |\mu|\cdot \v(\mu)}{\langle p_{-1} | \mathbf{h}_{m_1} | \cdots | \mathbf{h}_{m_n}\rangle_{L}}\\
&=&\frac{-2^{\chi/2}}{16 \pi^2 \cdot |\mu|\cdot \v(\mu)}
{[z_1^{m_1}\cdots z_n^{m_n}]\,\partial_2\Hfunc_n}|_{\mathbf{h}_\ell\mapsto \boldsymbol{\alpha}_\ell}\,.
\end{eqnarray*}
\end{theorem}
\begin{proof}
We compute $c_0^\spin$ as the limit of Hurwitz numbers of the torus of large degree:  \[c_0^\spin(\mu) \= \lim_{D\to \infty} \frac{3}{\pi^2} \frac{\sum_{d=1}^D c_{-1}^\circ(d,\Pi)}{\sum_{d=1}^D N_d^{\circ}(\Pi)}, \]
where $\Pi=((m_1,1,\ldots,1),\ldots,(m_n,1,\ldots,1))$, $N_d^{\circ}(\Pi)$ is the number of \emph{connected} torus covers of degree~$d$ with ramification profile $\Pi$, while $c_{-1}^\circ(d,\Pi)$ is the sum over those covers with $-1$st Siegel--Veech weight and with sign given by the parity. The proof of this formula is obtained from a direct transposition of the proof of the analogue result for $c_0(\mu)$ given by Proposition 17.1 of~\cite{CheMoeZag}.

By \cref{cor:weightedc0} and the inclusion-exclusion principle used to obtain connected cover counts from possibly disconnected ones, we have:
\begin{align*}
 c_\ell^\circ(d,\Pi) &\= 2^{\chi/2} [q^d] \langle p_\ell | \mathbf{f}_{m_1} | \cdots | \mathbf{f}_{m_n}\rangle, \\
 N_d^{\circ}(\Pi) & \underset{d\to \infty}{\sim} d^{|\mu|} \Vol(\mu),
\end{align*}
where the connected $q$-brackets were defined by~\eqref{eq:connectedq}. 
The growth rate of the connected bracket is determined by \cref{prop:growthN}.  Note $2-2g=\chi=n-|\mu|$. 
Therefore,
\begin{align*} c_0^\spin(\mu)  
&\=  \frac{3}{\pi^2}\lim_{D\to\infty} \frac{2^{\chi/2}[ p_{-1} | \mathbf{f}_{m_1} | \cdots | \mathbf{f}_{m_n}]_{D}}{\sum_{d=1}^D d^{|\mu|}\Vol(\mu)} \\
 &\=  \frac{3}{\pi^2}\lim_{D\to\infty} \frac{2^{\chi/2} \langle p_{-1} | \mathbf{f}_{m_1} | \cdots | \mathbf{f}_{m_n}\rangle_{L}\frac{D^{|\mu|+1}}{(|\mu|+1)!} (2\pi\ii)^{2g}}{\frac{D^{|\mu|+1}}{|\mu|+1}\Vol(\mu)} \\
&\= \frac{3}{ \pi^2}   \frac{2^{\chi/2}\langle p_{-1} | \mathbf{f}_{m_1} | \cdots | \mathbf{f}_{m_n}\rangle_{L}}{2|\mu|\, \v(\mu)}\prod_{i=1}^n m_i \\
& \=  \frac{3}{ \pi^2} \frac{2^{\chi/2} \langle p_{-1} | \mathbf{h}_{m_1} | \cdots | \mathbf{h}_{m_n}\rangle_{L}}{2|\mu|\, \v(\mu) },
\end{align*}
where the last equation follows as $\mathbf{f}_\ell-\mathbf{h}_\ell/\ell$ is of weight less than $\ell+1$. 
Moreover, by using \cref{lem:L}, \eqref{eq:hbracexpr}, and \cref{lem:comm}, we find
\begin{align*}
 \langle p_{-1} | \mathbf{h}_{m_1} | \cdots | \mathbf{h}_{m_n}\rangle_L &\= 
\frac{1}{\mathrm{Aut}(\mathbf{m})} \,[u_{m_1}\cdots u_{m_n}]\, \Psi^{H,\circ}_{-1}(\mathbf{u})_L \\ &\=
\frac{1}{\mathrm{Aut}(\mathbf{m})} \,[u_{m_1}\cdots u_{m_n}]\, \mathcal{C}^{H,\circ}_{-1}(\mathbf{u})_L \\ &\=
\frac{-1}{24}\frac{1}{\mathrm{Aut}(\mathbf{m})}\, [u_{m_1}\cdots u_{m_n}]\, \frac{\partial_2 \Psi^{H}(\mathbf{u})_L}{\Psi^{H}(\mathbf{u})_L} \\ &\=
\frac{-1}{24}\frac{1}{\mathrm{Aut}(\mathbf{m})}\, [u_{m_1}\cdots u_{m_n}]\, \partial_2 \Psi^{H,\circ}(\mathbf{u})_L \\ &\=
 \frac{-1}{24}\, [z_1^{m_1}\cdots z_n^{m_n}]\, \partial_2\Hfunc_n |_{\mathbf{h}_\ell\mapsto \boldsymbol{\alpha}_\ell}\,, 
 \end{align*}
 where $\mathrm{Aut}(\mathbf{m})$ denotes the cardinality of the stabilizer of the action of $\Sym_n$ on $(m_1,\ldots,m_n)$, and where we took the $L$-brackets of the sequences $\Psi^{H,\circ}_{-1}(\mathbf{u})$, $\mathcal{C}^{H,\circ}_{-1}(\mathbf{u})$, $\Psi^{H}(\mathbf{u})$ and $\Psi^{H,\circ}(\mathbf{u})$ defined in \cref{sec:growth}.
\end{proof}

%%%%%%%%%%%%%%%%%%%%%%%%%%%%%%%%
%%%%%%%%%%%%%%%%%%%%%%%%%%%%%%%%
\section{Twisted graphs and the boundary of the Hodge bundle}
%%%%%%%%%%%%%%%%%%%%%%%%%%%%%%%%
%%%%%%%%%%%%%%%%%%%%%%%%%%%%%%%%

{Fix $\mu,g,n$ as in the introduction.}
We will consider the following cohomology classes of the projectivized Hodge bundle:
\begin{itemize}[leftmargin=25pt]
\item the tautological class $\xi= c_1(\mathcal{O}(1)) \in H^2(\PP\oH_{g,n},\QQ)$;
\item the Poincar\'e-dual class of the locus of curves with a non-separating node $\delta_0\in H^2(\oM_{g,n},\QQ)$;
\item the Chern class of the cotangent line at the $i$th marking $\psi_i\in H^2(\oM_{g,n},\QQ)$ for all $1\leq i\leq n$.
\end{itemize}
For all $1\leq i\leq n$, we denote
\begin{align}\label{eq:beta} \beta_i= \xi^{2g-2} \cdot \left(\prod_{j\neq i} m_i \cdot \psi_i\right).
\end{align}
Let $X$ be a component of $\cH(\mu)$. We denote by $\PP\overline{X}$ the Zariski closure of the projectivization of $X$. In~\cite{Sau4},~\cite{CheMoeSauZag}, and~\cite{CosMoeZac} it was shown that 
$$
\v(X)=\int_{\PP\overline{X}} \beta_i\cdot \xi  
$$
 for all $1\leq i\leq n$ (see Notation~\ref{not} for the definition of $\v$ and $\v^\spin$). Here, we will consider the following intersection numbers for all $1\leq i\leq n$
 $$
 d_i(X)=\int_{\PP\overline{X}} \beta_i\cdot \delta_0.
 $$ 
  The purpose of the rest of the paper is to prove the following result.
\begin{theorem}\label{th:mainint} For all connected components $X$ of , and $1\leq i\leq n$, we have:
$$
c_i(X)\=\frac{-m_i}{4\pi^2} \cdot \frac{d_i(X)}{ \v(X)}.
$$
\end{theorem} 
For all $X$, the number $d_i(X)$ is independent of the choice of $i$ (see~\cite{CheMoeSauZag}). Thus, this theorem directly implies Theorem~\ref{th:cylinders}   and the following expression for the area Siegel--Veech (SV) constants of connected components, which is a new check that the class~$\beta_i$ represents the Kontsevich--Zorich cocycle (see~\cite{Kon1,EskKonZor}). 
\begin{theorem}\label{th:intspin} For all odd $\mu$ and all $1\leq i\leq n$, we have:
$$
c_0(X)= \frac{-1}{4\pi^2} \cdot\frac{d_i^\spin(X)}{\v(X)}.
$$
\end{theorem} 
As SV constants of hyperelliptic components are explicit and Theorem~\ref{th:cylinders} holds trivially for these components, we only need to consider the SV constants of strata and their weighting by their spin sign:
$$
d_i(\mu)=d_i(\cH(\mu)),\quad \text{ and }\quad d_i^\spin(\mu)=d_i\left(\cH(\mu)^\even\right)-d_i\left(\cH(\mu)^\odd\right).
$$
Besides, we will show that Theorem~\ref{th:mainint} holds when we set $i=1$ (the general case follows immediately by permuting the entries of $\mu$). We proceed in two steps:
\begin{enumerate}
\item In the present section we use intersection theory to express the numbers $d_1(\mu)$ and $d_{1}^\spin(\mu)$ in terms of the functions $\v, \v^\spin$ and intersection numbers in genus 0 (see~Proposition~\ref{pr:dint}).
\item  In the next section we use arguments of combinatorics to show that Proposition~\ref{pr:dint}  may be rewritten as the sum of the contributions of cylinder configurations in the sense of~\cite{EskMasZor} thus finishing the proof of Theorem~\ref{th:mainint}.
\end{enumerate}

%%%%%%%%%%%%%%%%%%%%%%%
\subsection{Twisted graphs} 
%%%%%%%%%%%%%%%%%%%%%%%
We recall here the definition of twisted graphs of~\cite{FarPan}. A {\em stable graph} is the data of $$\Gamma=(V,H,g:V\to \mathbb{Z}_{\geq 0},\iota:H\to H,\phi:H\to V,H^\iota\simeq [\![1,n]\!]),$$
where:
\begin{itemize}[leftmargin=25pt]
\item An element $v\in V$ is called a {\em vertex}. We denote by $g(v)$ the {\em genus} of $v$.
\item An element $h\in H$ is called an {\em half-edge}. We say $h$ is incident to $\phi(h)$, and write $h\mapsto v$ if $\phi(h)=v$. Moreover, we denote by $n(v)$ the {\em valency} of the vertex $v$, i.e.\ the number of half-edges incident to $v$. 
\item The function $\iota$ is an involution. The set~$E$ consist of cycles of length $2$ for $\iota$, which are called {\em edges}. 
\item The fixed points of $\iota$ are called {\em legs}. We write~$n$ for the number of legs, and identify the set of legs with the set $[\![1,n]\!]:=\{1,2,\ldots,n\}$.
\item For all vertices $v$ we have $2g(v)-2+n(v)>0$.
\item The graph~$(V,E)$ is connected.
\end{itemize}
The \emph{genus} of~$\Gamma$ is defined as
\[
g(\Gamma) = h^1(\Gamma)+\sum_{v\in V} g(v),\quad \text{ with } \quad h^1(\Gamma)=|E|-|V|+1.
\]
{An \emph{automorphism} of~$\Gamma$ consists of automorphisms of the sets~$V$ and~$H$ that leave invariant the data $g,\iota$ and $\phi$}. A stable graph is said to be \emph{of compact type} if $h^1(\Gamma) = 0$, i.e., if the graph is a tree.

\begin{definition}
A {\em twist} on a stable graph $\Gamma$ is a function $m:H\to {\ZZ}$ satisfying the following conditions:
\begin{itemize}[leftmargin=25pt]
\item For all $v\in V$, we denote by $\mu(v)$ the vector of twists at half-edges incident to $v$. We impose 
$$|\mu(v)|=\sum_{h\mapsto v} m(h) =  2g(v)-2+n(v).$$
\item   If $e=(h,h')$ is an edge of $\Gamma$ from $v$ to $v'$, then we have $m(h)=-m(h')$.
\item There exists a partial order~$\geq$ on~$V$ such that for all vertices $v,v'$ connected by an edge $(h,h')$ we have  $(v\geq v') \Leftrightarrow  (m(h)\geq 0)$.
\end{itemize}
For an edge $e=(h,h')$ from $v$ to $v'$, we call $m_e=|m(h)|$ the \emph{twist at the edge}. If $m_e=0$ (or, equivalently, if $v\geq v'$ and $v'\geq v$) then we will say that the edge is {\em horizontal}. 

A \emph{twisted graph}  is a pair $(\Gamma,m)$, where $m$ is a twist on~$\Gamma$. It is said {\em compatible} with an integral vector $\mu$ of length $n$, if the twist at the $i$th leg is equal to $m_i$. 
\end{definition}

Most of the twisted graphs that will be considered will be in the following set.

\begin{definition}
A twisted graph $(\Gamma,m)$ is a {\em graph of rational type} if it is of compact type and there exists a partition of the set of vertices $V(\Gamma)=R(\Gamma)\sqcup D(\Gamma)$ satisfying:
\begin{itemize}[leftmargin=25pt]
\item The twists at the legs are non-negative.
\item $g(v)=0 \Leftrightarrow v\in R(\Gamma)$ (the set of {\em rational vertices}).
\item If $v\in D(\Gamma)$ (the set of {\em decorations}), then all half-edges incident to $v$ have positive twist. 
\end{itemize} 
\end{definition}

\begin{remark} In this definition, the twist and the partition of the set of vertices are uniquely determined by the underlying stable graph and the twists at the legs in $[\![1,n]\!]$. In particular, an automorphism of $\Gamma$ automatically respects the twist function. Thus, to keep the notation simple, we denote by $\Gamma$ a graph of rational type. 
\end{remark}

\begin{definition}
A twisted graph $(\Gamma,m)$ is a {\em bicolored graph} if there is a partition of the set of vertices $V(\Gamma)=V_0\sqcup V_{-1}$ such that all edges connect a vertex $v\in V_0$ to a vertex $v'\in V_{-1}$ with $v>v'$.

A bicolored graph is a {\em rational backbone graph}  if it is of compact type, has a unique vertex in $V_{-1}$ of genus 0 which carries the first marking, and all vertices of~$V_0$ have positive genus (note that a rational backbone graph is of rational type and satisfies $D(\Gamma)=V_0$ and $R(\Gamma)=V_{-1}$).\end{definition}

%%%%%%%%%%%%%%%%%%%%%%%
\subsection{Boundary of strata of differentials}
%%%%%%%%%%%%%%%%%%%%%%%

If $\mu$ is a vector of (not necessarily positive) integers of length~$n$ {with $|\mu|=2g-2+n$}, then we denote by $\H(\mu)$ the moduli space of objects $(C,x,\eta)$, where $C$ is smooth and $\eta$ is a meromorphic differential with ${\rm ord}_{x_i}(\eta)=m_i-1$ for all $1\leq i\leq n$. This space is canonically embedded in the vector bundle 
$$\pi_*\omega_{\oC_{g,n}/\oM_{g,n}}(p_1\cdot D_1+\ldots+p_n \cdot D_n)
$$
where $\pi:\oC_{g,n}\to \oM_{g,n}$ is the universal curve, $D_i$ is the divisor associated to the $i$th marking, and $p_i$ is a positive integer bigger than $-m_i$ for all $1\leq i\leq n$. The {\em incidence variety compactification} $\PP\oH(\mu)$ is the Zariski closure of $\PP\H(\mu)$ in the projectivization of the above vector bundle. The geometry of $\PP\oH(\mu)$ does not depend on the choice of the $p_i$'s and was described in~\cite{BCGGM}. 

\subsubsection{Residue conditions} If $\mu$ has $r$ non-positive entries then we denote by $\mathfrak{R}(\mu)$ the subspace of $\CC^r$ with sum equal to $0$. If $R$ is a linear subspace of $\mathfrak{R}(\mu)$, then we denote by $\H(\mu,R)\subset \H(\mu)$ and $\PP\oH(\mu,R)\subset \PP\oH(\mu)$ the space of differentials with residues in $R$ (up to a scalar in the projectivized case).

\subsubsection{Boundary components of $ \PP\oH(\mu)$}

We recall that a non-trivial stable graph $\Gamma$ determines a boundary component of the moduli space of curves 
$$\zeta_\Gamma:\oM_{\Gamma}=\prod_{v\in V(\Gamma)} \oM_{g(v),n(v)} \to \oM_{g,n}$$ 
with $g=g(\Gamma)$ and $n$ the number of markings. 
 A bicolored graph $(\Gamma,m)$ determines two moduli spaces:
\begin{eqnarray*}
\H(\Gamma,m,R)_0&=& \prod_{v \in V_0} \H(\mu(v),R_v)\\
\H(\Gamma,m,R)_{-1}&=& \prod_{v \in V_{-1}} \H(\mu(v),R_v),
\end{eqnarray*}
where for all $v$
%, the vector $\mu(v)$ stands for the vector of twists at half-edges incident to $v$, and 
the vector space $R_v$ is defined by the so-called {\em global residue condition} defined in~\cite{BCGGM}. Moreover, it determines a morphism
$$
\zeta_{(\Gamma,m)}: \PP\H(\Gamma,m,R)_0\times \PP\H(\Gamma,m,R)_{-1} \to \PP\oH(\mu,R).
$$
Denote by $\PP\oH(\Gamma,m,R)$ the Zariski closure of the image of this morphism. 

Besides, if $(\Gamma,m)$ is a twisted graph with exactly one horizontal edge, then we denote by $\PP\oH(\Gamma,m,R)\subset\PP\oH(\mu,R)$ the subspace of differentials whose underlying curve sits in the image of $\oM_{\Gamma}$. With these notations at hand, $\PP\oH(\mu,R)$ is the union of the $\PP\oH(\Gamma,m,R)$ for $(\Gamma,m)$ bicolored graphs and twisted graphs with exactly one horizontal edge compatible with~$\mu$. 

\subsubsection{Generalization of spin parity} If $\mu$ is odd, then the parity of a point $(C,x,\eta)\in \H(\mu)$ is the parity of $h^0\bigl(C,\mathcal{O}\bigl(\frac{m_1-1}{2}+\ldots+\frac{m_n-1}{2}\bigr)\bigr).$ 

Besides, if $\mu$ has only odd entries apart from the first two which are equal to 0, and $R\subset \mathfrak{R}(\mu)$ is the vector space defined by $r_1+r_2=0$, then the parity of a point in $(C,x,\eta)\in \H(\mu,R)$ is defined as the parity of the differentials in the desingularization of the node created by attaching the two poles of order~$1$ (see~\cite{Boi} for the details of this construction). Then, as in the holomorphic case, we denote by 
$$[\PP\oH(\mu,R)]^\spin=[\PP\oH(\mu,R)]^\even-[\PP\oH(\mu,R)]^\odd.$$

%%%%%%%%%%%%%%%%%%%%%%%
\subsection{Functions defined by intersection numbers in genus 0}
%%%%%%%%%%%%%%%%%%%%%%%
Here, we define three functions $f$, $\varphi$, and $\varphi^{\spin}$ as intersection numbers on strata of differentials of genus 0. We also show how to compute these functions recursively. The results collected in this section will only be used in Section~\ref{sec:chains} to manipulate the sums indexed by different families of graphs of rational type. 

\subsubsection{The $f$-function} Let $\mu=(m_1,m_2,\ldots)\in \ZZ^n$ be vector of length $n\geq 3$ such that: no entry of $\mu$ is equal to 0,  $m_1$ and $m_2$ are positive, and $|\mu|=n-2$, where $|\mu|=\sum_{i=1}^n m_i$. Then, we define
$$
f(\mu) \= \int_{\PP\oH(\mu,\{0\})}  \!\!\prod_{\substack{ i\geq 3\\ \text{ s.t.\ }m_i>0}}\!\! m_i\, \psi_i\,,
$$
where $\{0\}$ stands for the trivial vector space. Note that this definition differs from the one in~\cite{CheMoeSauZag} by a product of the~$m_i$.

This function may be computed inductively by using twisted graphs. We fix an index $i$ in $[\![2,n]\!]\setminus\{3\}$. We say that a twisted graph~$\Gamma$ of genus 0 is of {\em type~$f_{i}$} if it has exactly two vertices~$v_{0}>v_{-1}$, one (non-horizontal) edge~$e$ and satisfies:
\begin{itemize}[leftmargin=25pt]
\item The legs $i$ and $1$ are on one of the vertices, and the leg $3$ is on the other one.
\item The vertex $v_0$ (respectively $v_{-1}$) has exactly 1 (respectively 2) of the legs with index in $\{1,i,3\}$.
\end{itemize}
Moreover, write $\Gamma\vdash f_i$ to denote that $\Gamma$ is of type $f_i$ and compatible with $\mu$. Then we set
$$
f(\Gamma)=  f(\mu(v_{-1}))\cdot f(\mu(v_{0})),
$$ 
where we order the entries of $\mu(v_{-1})$ and $\mu(v_{0})$ in such a way that the first entries of $\mu(v_{-1})$ are the $m_j$'s for the values $j\in \{1,i,3\}$ and incident to $v_{-1}$, while the first entries of $\mu(v_{0})$ are $m_e$ and $m_j$ for the last value $j\in \{1,i,3\}$.

The following lemma generalizes Proposition~2.1 of \cite{CheMoeSauZag} (case $i=2$ of the recursion formula).
\begin{lemma}
\label{lem:indf} If $m_i<0$ for $i>2$, then we have:
$$
f(\mu)\= (n-3)!\, [t^{m_2}]  \prod_{i >2}t\frac{1-t^{-m_i}}{1-t}
$$
where $[t^{m_2}]$ stands for $m_2$-coefficient in the variable~$t$. Moreover, if $m_3>0$, then for all $i\in [\![2,n]\!]\setminus\{3\}$ we have the following relation:
$$
f(\mu) \= m_3  \sum_{\Gamma  \vdash f_i}  f(\Gamma),
$$
where the sum is over all twisted graphs of type~$f_i$ compatible with $\mu$. 
\end{lemma}
\begin{proof}
We use the same approach as in~\cite{CheMoeSauZag}. For all $i$, we have
$$
\psi_{3}= \sum_{\Sigma \subset \{1,\ldots,n\}\setminus \{1,i,3\}} \zeta_{\Gamma_\Sigma *}(1) \in H^2(\oM_{g,n},\QQ),
 $$
 where $\Gamma_\Sigma$ is the graph with one edge and one of the vertices carries the markings in $\Sigma\cup\{1,i\}$ while the other %carry the 
 vertex carries the other markings. Then the intersection of $\zeta_{\Gamma_\Sigma *}(1)$ with 
 $$\H(\mu,\{0\})\cdot \!\!\prod_{\substack{i\geq 4\\ \text{ s.t.\ }m_i>0}}\!\! m_i\, \psi_i$$ is non-trivial if and only if the unique structure of twisted graph on the stable graph defining $\delta_{\Sigma}$ satisfies the constraints of the type $f_i\mspace{1mu}$. In that case, this intersection is exactly given by $f(\Gamma_\Sigma)$ (see~\cite{CheMoeSauZag}).
\end{proof}

We use this lemma to show another identity satisfied by the function $f$. For $i\in [\![3,n]\!]$, we say a twisted graph is of {\em type~$f_i'$} if it satisfies the same conditions as for the type~$f_i$, but interchanging the roles of the markings $2$ and $3$. Its contribution is also determined by the same formula as for graphs of type~$f_i$ by interchanging the roles of the markings $2$ and $3$.

\begin{lemma}\label{lem:exch0} If $m_1, m_2, m_3$ are positive, then we have:
$$
\sum_{\Gamma \vdash f_2} (m_{e}+m_3)\, f(\Gamma)\=\sum_{\Gamma \vdash f_3'} (m_e+m_2)\, f(\Gamma).
$$
\end{lemma}
\begin{proof} 
We have the following identity in $H^*(\oM_{0,n+2},\QQ)$:
\begin{align*}
(m_3\psi_3 - & m_2\psi_2)\,  p_*\left[\PP\oH(\mu,\{0\}) \right] 
\\
&\= \sum_{\Gamma \in {\rm Bic}_{2,3}}\!\! m_{E(\Gamma)}\, p_*[\PP\oH(\Gamma,m,\{0\})] \meno  \sum_{\Gamma\in {\rm Bic}_{3,2}}\!\! m_{E(\Gamma)}\, p_*[\PP\oH(\Gamma,m,\{0\})] ,
\end{align*}
where  ${\rm Bic}_{j,j'}$ is the set of bicolored graphs with the leg $j$ incident to a vertex in $V_{-1}$ and $j'$ is incident to a vertex in $V_0$ (see~\cite[Theorem~6]{Sau}), and $m_{E(\Gamma)}=\prod_{e\in E(\Gamma)} m_e\mspace{1mu}$. We intersect this identity with 
 \[\!\!\prod_{\substack{i\geq 4\\ \text{ s.t.\ }m_i>0}}\!\! m_i\, \psi_i\,.\] 
 On the left-hand side, we use Lemma~\ref{lem:indf} as sums over graphs of type~$f_2$ or $f_3'$. On the right-hand side, only the graphs of type~$f_2$ and $f_3'$ give a non-trivial contribution on the right-hand side (see~\cite[Section~2.4]{CheMoeSauZag}. All in all, we obtain the following expression:
\begin{eqnarray*}
&&\sum_{\Gamma \vdash f_2} m_3\, f(\Gamma) \meno \sum_{\Gamma \vdash f_3'} m_2\, f(\Gamma) 
\= \sum_{\Gamma \vdash f_3'} m_e\, f(\Gamma) \meno  \sum_{\Gamma \vdash f_2} m_e\, f(\Gamma),
\end{eqnarray*}
which is the desired identity. 
\end{proof}

\subsubsection{The $\varphi$-function} Now, let $\mu$ be a vector of length $n+2$ satisfying: $m_1$ is positive, $m_{n+1}$ and $m_{n+2}$ are non-positive, no other entry is zero, and   $|\mu|=n$. Then, we denote
\begin{eqnarray*}
\varphi(\mu)\= \int_{\PP\oH(\mu,R_{1,2})} \!\!\prod_{\substack{ i\geq 2\\ \text{s.t.\ }m_i>0}}\!\! m_i\, \psi_i\,,
\end{eqnarray*}
where $R_{1,2}$ is the vector space defined by $r_{n+1}+r_{n+2}=0$ while all other residues are equal to $0$.  Moreover, if all entries of $\mu$ are odd apart from $m_{n+1}=m_{n+2}=0$, we denote:
\begin{eqnarray*}
\varphi^\spin (\mu)\=\int_{[\PP\oH(\mu,R_{1,2})]^{\spin}}  \!\!\prod_{\substack{ i\geq 2\\ \text{s.t.\ } m_i>0}}\!\! m_i\, \psi_i\,.
\end{eqnarray*}

\begin{figure}
\begin{tabular}{|c|c|}
\hline
Type of twisted graph & Contribution \\
\hline 
{ Type~$f_2$} & \xymatrix@C=1em@R=0.5 em{
\mathbf{f} \\
 *+[Fo]{v_0}  \ar@{-}[d] \ar@{-}[rd]^{\!\!\!\!+m_e}_{ -m_e\!\!\!\!} & \mathbf{f}  \\
m_3& *+[Fo]{\! v_{-1}\!\! } \ar@{-}[d] \ar@{-}[rd] 
\\
&m_1& m_2
} \\
\hline Type~$\varphi_1$ &  \xymatrix@C=1em@R=0.5 em{
\mathbf{f} \\
 *+[Fo]{v_0}  \ar@{-}[d] \ar@{-}[rd]^{\!\!\!\!+m_e}_{ -m_e\!\!\!\!} &  \boldsymbol{\varphi} \\
m_2& *+[Fo]{\! v_{-1}\!\! } \ar@{-}[d] 
\\
&m_1&
} \\
Type~$\varphi_2$ &  \xymatrix@C=1em@R=0.5 em{
&\boldsymbol{\varphi} & \boldsymbol{\varphi}& \\
m_{n+1} & *+[Fo]{v_1}  \ar@{-}[d] \ar@{-}[l] \ar@{-}[r]^{\!\!\!\!\!0}_{ 0\!\!\!\!\!} &  *+[Fo]{v_{2}}  \ar@{-}[d] \ar@{-}[r]& m_{n+2}\\
&m_1 & m_2  
}\\
Type~$\varphi_3$ &  \xymatrix@C=1em@R=0.5 em{
&&\mathbf{f} && \\
&\boldsymbol{\varphi}\,\,\,\,\,\,\, &*+[Fo]{v_0}   \ar@{-}[ld]^{\!\!\!\!-m_{e_1}}_{ +m_{e_1}\!\!\!\!} \ar@{-}[rd]^{\!\!\!\!+m_{e_2}}_{ -m_{e_2}\!\!\!\!} & \,\,\,\,\,\,\,\boldsymbol{\varphi}& \\
m_{n+1} & *+[Fo]{v_1}  \ar@{-}[d] \ar@{-}[l]  &&  *+[Fo]{v_{2}}  \ar@{-}[d] \ar@{-}[r]& m_{n+2}\\
&m_1 && m_2  
}\\
 \\
\hline
\end{tabular}
\caption{The types of twisted graphs involved in the induction formulas defining $f$ and $\varphi$. Only the twists at the first legs and the half-edges of the edge are indicated. Besides,  the letters $\mathbf{f}$ or $\boldsymbol{\varphi}$ above each vertex indicate which function is used to define the contribution of the graph.\label{fig:table}}
\end{figure}

We will need three types of graphs of rational type of genus~$0$ to compute~$\varphi$ and~$\varphi^\spin$ inductively (see, also, \cref{fig:table}). A graph of rational type~$\Gamma$ is of
\begin{itemize}[leftmargin=25pt]
\item {\em Type~$\varphi_1$} if it has one edge~$e$, the legs~$1$ and $n+1$ are incident to a vertex $v_{-1}$ which is lower than the vertex $v_0$ carrying the leg~$2$. Then, we set
\begin{eqnarray*}
\varphi(\Gamma)&=& \varphi(\mu(v_{-1}))\cdot f(\mu(v_{0})),\\
\varphi^\spin(\Gamma)&=& \varphi^\spin(\mu(v_{-1}))\cdot f(\mu(v_{0})),
\end{eqnarray*}
where  $\mu(v_{-1})=(m_1,\ldots,m_{n+1},{-}m_e)$ and $\mu(v_{0})=(m_2,m_e,\ldots)$.
\item  {\em Type~$\varphi_1'$} if it has one edge~$e$, the legs~$2$ and $n+2$ are incident to a vertex $v_{-1}$ which is lower than the vertex $v_0$ carrying the legs~$1$ and $n+1$. Then, we set
\begin{eqnarray*}
\varphi(\Gamma)&=& \varphi(\mu(v_{-1}))\cdot f(\mu(v_{0}))
\end{eqnarray*}
where  $\mu(v_{-1})=(m_2,\ldots,m_{n+2},{-}m_e)$ and $\mu(v_{0})=(m_1,m_e,\ldots)$. (Note that this configuration may only occur if $m_{n+1}$ is negative thus we do not need to define $\varphi^\spin$ for such a graph).
\item {\em Type~$\varphi_2$} if it has one edge, which is horizontal, the legs~$1$ and $n+1$ are incident to a vertex $v_1$ and the legs~$2$ and $n+2$ to the vertex $v_2$. Then, we set
\begin{eqnarray*}
\varphi(\Gamma)&=& \varphi(\mu(v_1))\cdot \varphi(\mu(v_{2})),\\
\varphi^\spin(\Gamma)&=& \varphi^\spin(\mu(v_1))\cdot \varphi^\spin(\mu(v_{2})),
\end{eqnarray*} 
where  $\mu(v_{1})=(m_1,\ldots,m_{n+1},0)$ and $\mu(v_{0})=(m_2,\ldots,0,m_{n+2})$.
\item {\em Type~$\varphi_3$} if it has two edges, both not horizontal, and the legs~$1$ and $n+1$ are incident to a vertex $v_1$, the legs~$2$ and $n+2$ to a vertex $v_2$ which are both connected to an upper vertex~$v_0$ with edges $e_1$ and $e_2$. Then, we set
$$
\varphi(\Gamma) \= (m_{e_1}+m_{e_2}) \cdot  \varphi(\mu(v_1))\cdot f(\mu(v_{0}))\cdot \varphi(\mu(v_{2})),
$$
where  $\mu(v_{1})=(m_1,\ldots,m_{n+1},{-}m_{e_1})$,  $\mu(v_{2})=(m_2,\ldots,{-}m_{e_2},m_{n+2})$ and $\mu(v_{-1})=(m_{e_1}, m_{e_2},\ldots)$.
\end{itemize}\vspace{10pt}

The function $\varphi$ may be computed by the following lemma.
\begin{lemma}\label{lem:indg}
If $m_i<0$ for all $1<i\leq n$, then: 
\begin{eqnarray*}
\varphi (\mu) \= (n-1)!\, \prod_{i=2}^{n} -m_i, & \text{ and }& \varphi^\spin (\mu)= -(n-1)!.
\end{eqnarray*}
If $m_2>0$, then we may compute $\varphi$ and $\varphi^\spin$ recursively:
\begin{eqnarray*}
\varphi(\mu)&=& m_2\, \sum_{\Gamma \vdash \varphi_1,\varphi_1', \varphi_2, \varphi_3} \varphi(\Gamma), \\
\varphi^\spin (\mu)&=& m_2\,\sum_{\Gamma \vdash \varphi_1, \varphi_2}\varphi^\spin(\Gamma),
\end{eqnarray*}
where the sums are over graphs of type~$\varphi_1,\varphi_1',\varphi_2,$ or $\varphi_3$ compatible with $\mu$.
\end{lemma}

\begin{proof} The base cases of the lemma are given by Proposition 3.8 and Lemma 4.8 of~\cite{CheChe1}. We use the same strategy to prove the induction formula as the one used for $f$. We use the following relation in $H^2(\oM_{0,n+2},\QQ)$:
$$
\psi_{2}= \sum_{\Sigma \subset \{3,\ldots,n+2\}\setminus\{n+1\}} \delta_\Sigma \in H^2(\oM_{0,n+2},\QQ)
 $$
where $\delta_{\Sigma}=\zeta_{\Gamma_{\Sigma}*}(1)$, and  $\Gamma_{\Sigma}$ for the unique graph with one edge and one vertex carrying the legs in $\{1,n+1\}\cup \Sigma$ (while the other one carries the others). 

Let $\Sigma$ be a set appearing in the expression of $\psi_2$. The schematic intersection of $\delta_{ \Sigma}$ with  $\PP\oH(\mu, R_{1,2})$ is the union of all divisors of $\PP\oH(\mu, R_{1,2})$ defined by bicolored graphs or twisted graphs with one horizontal edge and which have an edge which separates $\{x_1,n+1\}\cup \Sigma$ from the other legs. To compute the intersection number of $(\delta_{\Sigma}\cdot [\PP\oH(\mu, R_{1,2})])$ with  
$$\!\!\prod_{\substack{ i\geq 2\\ \text{s.t.\ }m_i>0}}\!\! m_i\, \psi_i$$ 
we only need to consider the twisted graphs that do not vanish once we push forward the class $[\PP\oH(\Gamma,m,R_{1,2})]$ along $p:\PP\oH(\mu, R_{1,2})\to \oM_{0,n+2}$. 
This is the case only for the types of graphs $\varphi_1, \varphi_1',\varphi_2$ and $\varphi_3$. {Namely}, as there are exactly two poles with opposite residues, there may be only two vertices per level. Indeed:
\begin{itemize}[leftmargin=25pt]
\item If $(\Gamma,m)$ has a unique horizontal edge, then the two poles $n+1$ and $n+2$ should be carried by the two distinct vertices. Otherwise, the residue at the edge would vanish and the graph would define a space of dimension smaller than ${\rm dim}(\PP\oH(\mu, R_{1,2}))-1$. 
\item If $(\Gamma,m)$ is a bicolored graph, then $V_{-1}$ may have up to 2 vertices and $V_0$ only 1. If $V_{-1}$ has two vertices, then necessarily, each vertex carries one of the poles with non-vanishing residue. This gives the type~$\varphi_1, \varphi_1'$ and $\varphi_3\mspace{1mu}$.
\end{itemize} 

The multiplicity of the graphs of type~$\varphi_1,\varphi_1'$ and $\varphi_2$ is one as for the graphs of type~$f$. Thus we only need to show that a graph $(\Gamma,m)$ of type~$\varphi_3$ contributes trivially to the function $\varphi^{\spin}$ and with multiplicity $m_{e_1}+m_{e_2}$ to $\varphi$. 

To do this we choose $i=1$ or $2$. Then the edge $e_i$ determines a set $\Sigma$ as: the set of legs incident to the vertex $v_1$ if $i=1$, and the set of legs incident to the vertices $v_1$ and $v_0$ if $i=2$. In both cases we will show that $$\delta_\Sigma\cdot [\PP\oH(\mu, R_{1,2})]= m_{3-i} \,\PP\oH(\Gamma,m, R_{1,2})],$$ while  $\delta_\Sigma\cdot [ \PP\oH(\mu, R_{1,2})]^\spin=0$. Indeed, a generic point $y\in \PP\H(\Gamma,m, R_{1,2})$ has neighborhood in $\PP\oH(\mu, R_{1,2})$ given by $\Delta \times G \times U,$
 where $U$ is neighborhood of $y$ in $\PP\H(\Gamma,m, R_{1,2})$, $G$ is a discrete set of cardinality ${\rm gcd}(m_{e_1},m_{e_2})$, and $\Delta$ is an open disk of $\CC$ containing $0$  parametrized by {some parameter~}$\epsilon$ (see~\cite[Lemma~5.6]{Sau}).  Moreover, the neighborhood of the node corresponding to $e_i$ in the universal curve $\mathcal{C}\to \Delta \times G \times U$  is given by $z \times w= \epsilon^{{\rm lcm}(m_{e_1},m_{e_2})/m_{e_i}}$. Therefore, the intersection of $\PP\oH(\mu, R_{1,2})$ with $\delta_{\Sigma}$ is equal to 
 $$
 {\rm gcd}(m_{e_1},m_{e_2}) \times \frac{{\rm lcm}(m_{e_1},m_{e_2})}{m_{e_i}} \= \frac{m_{e_1}m_{e_2}}{m_{e_i}},
 $$
 which is the expected contribution for the function $\varphi$. 
 
 For the function $\varphi^{\spin}$, we consider the parity of the generic element in 
 $$U\times \{\gamma\} \times\left(\Delta\setminus\{0\}\right) $$ 
 for an element in {$\gamma\in G$}. As both vertices have even twists, half of the elements of $G$ give even or odd parity (see~\cite[Proposition~5.2]{CosSauSch}). Therefore the intersection of $\delta_\Sigma$ with $[\PP\oH(\mu, R_{1,2})^\even]-[\PP\oH(\mu, R_{1,2})^\odd]$ is trivial.
\end{proof}

%%%%%%%%%%%%%%%%%%%%%%%
\subsection{Intersection of strata with $\delta_0$} 
%%%%%%%%%%%%%%%%%%%%%%%

From now on $\mu$ is a vector of \emph{positive} entries. We recall from~\cite{CheMoeSauZag} the following induction formula for the normalized volume (see Notation~\ref{not}). 
\begin{proposition}\label{pr:volint}
If $n\geq 2$, then we have the following relation:
\begin{eqnarray*}
\v(\mu)&=&\sum_{\Gamma \in {\rm BB}(\mu)_2} f(\mu(v_{-1})) \frac{\prod_{e\in E(\Gamma)}m_e}{|{\rm Aut} (\Gamma)|} \prod_{v\in V_0} \v(\mu(v)),\\
\v^\spin(\mu)&=&\sum_{\Gamma \in {\rm BB}(\mu)_2}  f(\mu(v_{-1})) \frac{\prod_{e\in E(\Gamma)} m_e}{|{\rm Aut} (\Gamma)|}\prod_{v\in V_0} \v^\spin(\mu(v)),
\end{eqnarray*}
where ${\rm BB}(\mu)_2$ is the set of rational backbone graphs compatible with $\mu$ such that the second leg is incident to the vertex $v_{-1}$ in $V_{-1}$, and where $\mu(v_{-1})=(m_1,m_2,\ldots)$ is the vector of twists at half-edges incident to $v_{-1}$. % starting with $m_1$ and $m_2$.
\end{proposition}
Here we will prove the following expression of the functions $d_1$ and $d_1^\spin$ in terms of the volume function.
\begin{proposition}\label{pr:dint}
For all $\mu$, we have
\begin{eqnarray*}
d_1(\mu)&=&\sum_{\Gamma \in {\rm BB}(\mu)_0} \varphi(\mu(v_{-1})) \frac{\prod_{e\in E(\Gamma)} m_e}{2\, |{\rm Aut} (\Gamma)|}\prod_{v\in V_0} \v(\mu(v)),\\
d_1^\spin(\mu)&=&\sum_{\Gamma \in {\rm BB}(\mu)_0}  \varphi^\spin(\mu(v_{-1})) \frac{\prod_{e\in E(\Gamma)} m_e}{2\, |{\rm Aut} (\Gamma)|} \prod_{v\in V_0} \v^\spin(\mu(v)),
\end{eqnarray*}
where ${\rm BB}(\mu)_0$ is the set of rational backbone graphs compatible with $(m_1,\ldots,m_n,0,0)$ with the legs $n+1$ and $n+2$ incident to $v_{-1}$, and $\mu(v_{-1})=(m_1,\ldots,0,0)$ is the vector of twist at half-edges incident to $v_{-1}$. %starting with $m_1$ and finishing with two zeros.
\end{proposition}

\begin{proof} To prove this proposition we will proceed in two steps: (i)~we express ${\delta_0\cdot [\PP\oH(\mu)]}$ in terms of boundary component of $\PP\oH(\mu)$; (ii)~we use the results of \cite{Sau} to compute the intersection of $\xi$ with these boundary components.\bigskip

\step{Intersection of strata with $\delta_0$.} Let $X$ be connected component of $\PP\oH(\mu)$. Up to loci of co-dimension 2, the schematic intersection of $\PP\overline{X}$ and $\delta_0$ is contained in the union of the spaces $\PP\oH(\Gamma,m)$, where $(\Gamma,m)$ is either the unique graph with one non-separating horizontal edge, or a bicolored graph with  $h^1(\Gamma)>0$. We will show that graphs of the second type contribute trivially to the integral of $\beta_1$ {(see~\eqref{eq:beta} for the definition of $\beta_1$)}.

Let $(\Gamma,m)$ be a bicolored graph with $h^1(\Gamma)>0$ and let $X$ be a connected component of $\H(\Gamma,m)$. We have 
$$
\beta_1\cdot [\PP\overline{X}] \= p_*(\xi^{2g-2}\cdot [\PP\overline{X}]) \cdot \prod_{i=2}^n m_i\,\psi_i\,,
$$ 
where $p: \PP\oH(\mu)\to \oM_{g,n}$ is the forgetful morphism of the differential. We have $ p_*(\xi^{2g-2}\cdot [\PP\overline{X}])=0$ by~\cite[Proposition~3.10]{CheMoeSauZag} (actually the proposition is given there for the complete stratum $\PP\oH(\mu)$, but the arguments can be transposed directly for all connected components). Therefore, as $\delta_0$ intersect transversally along $(\Gamma_0,m_0)$, the unique graph with one non-separating horizontal edge, we have:
$$
d_1(\mu)\= \int_{\PP\oH(\Gamma_0,m_0)} \beta_i \= \frac{1}{2} \int_{\PP\oH(\mu+(0,0))} \xi^{2g-2}\prod_{i=2}^n m_i\,\psi_i
$$
(the factor $1/2$ comes from the automorphism group of $(\Gamma_0,m_0)$).
Moreover, we also get the spin analogue:
$$
d_1(\mu) \= \frac{1}{2} \int_{[\PP\oH(\mu+(0,0))]^\spin} \xi^{2g-2}\prod_{i=2}^n m_i\,\psi_i\,.
$$
\bigskip

\step{Expression of $\xi$ on $\PP\oH(\mu+(0,0))$.} We use \cite{Sau} to write $\xi$ as a linear combination of boundary divisors. Indeed,
$$
\xi \cdot [\PP\oH(\mu+(0,0))] \= [\PP\oH(\mu+(0,0), \{0\})] + \text{ boundary terms.}
$$
However, $\PP\oH(\mu+(0,0), \{0\})$ is empty as there can be no pole of order exactly one with vanishing residue. The boundary terms are supported on all bicolored graphs $(\Gamma,m)$ with the two poles incident to vertices of $V_{-1}$. We will show that if $X$ is a connected component of $\H(\Gamma,m)$ which is not in ${\rm BB}(\mu)_0$, then $\int_{\PP\overline{X}} \xi^{2g-3}\prod_{i=2}^n \psi_i$ vanishes.

First, if we assume that $(\Gamma,m)$ is not a rational backbone graph, then $p_*(\xi^{2g-3}\cdot [\PP\overline{X}] = 0$ by the same arguments used to prove Proposition~3.10 of~\cite{CheMoeSauZag}. Thus we choose a {rational backbone} graph $(\Gamma,m)$. Then $\H(\Gamma)_{-1}$ has dimension $n_{-1}$, where $n_{-1}$ is the number of legs in $[\![1,n]\!]$ incident to the vertex of $V_{-1}$. Thus its projectivization has dimension $n_{-1}-1$, and we obtain
$$
\int_{\PP\oH_{(\Gamma,m)} }\prod_{i=2}^n \psi_i = 0
$$
unless $1$ is incident to the vertex of $V_0$. Putting everything together, we obtain the following result:
$$
\xi \cdot [\PP\oH(\mu+(0,0))]\= \sum_{\Gamma\in {\rm BB}(\mu)_0} \left(\prod_{e\in E(\Gamma)} m_e\right)\cdot [\PP\oH{(\Gamma)}]  + \Delta
$$
where $\int_{\Delta} \xi^{2g-3}\prod_{i=2}^n \psi_i=0$. Moreover, if $(\Gamma,m)\in {\rm BB}(\mu)_0$, then we have
$$
 \int_{\PP\oH(\Gamma,m)} \xi^{2g-3}\prod_{i=2}^n m_i\psi_i \= \frac{\varphi(\mu(v_{-1}))}{|{\rm Aut}(\Gamma)|}  \prod_{v\in V_0} \v(\mu(v)).
$$ 
 This is the desired expression of $d_1(\mu)$. Also, if $\mu$ is odd, then we have
$$
 \int_{[\PP\oH(\Gamma,m)]^\spin} \!\!\!\!\!\!\!\!\!\!\!\!\xi^{2g-3}\prod_{i=2}^n m_i\psi_i \= \frac{\varphi^\spin(\mu(v_{-1}))}{|{\rm Aut}(\Gamma)|}  \prod_{v\in V_0} \v^\spin(\mu(v)).
$$ 
Indeed, a backbone graph is of compact type thus the parity of a generic element of $\H(\Gamma,I)$ is defined as the product of the parities of the elements in $\H(\mu(v),R_v)$ for all vertices of $\Gamma$ (see~\cite{Cor}). 
\end{proof}

%%%%%%%%%%%%%%%%%%%%%%%%%%%%%%%%
%%%%%%%%%%%%%%%%%%%%%%%%%%%%%%%%
\section{Chains and cylinder configurations}\label{sec:chains}
%%%%%%%%%%%%%%%%%%%%%%%%%%%%%%%%
%%%%%%%%%%%%%%%%%%%%%%%%%%%%%%%%

In this section we use Proposition~\ref{pr:dint} to express $d_1(\mu)$ and $d_{1}^\spin(\mu)$ as sums over cylinder configuration. 

%%%%%%%%%%%%%%%%%%%%%%%
\subsection{Chains}
%%%%%%%%%%%%%%%%%%%%%%%

Here we define a category of graphs called chains to encode cylinder configurations in the spirit of~\cite{CheChe1}.

\begin{definition}
 A {\em chain} is a graph of rational type~$\Gamma$ with $n+2$ legs and a partition $D(\Gamma)=F({\Gamma})\sqcup P(\Gamma)$  satisfying the following constraints:
\begin{itemize}[leftmargin=25pt]
\item $m(n+1)=m(n+2)=0$, and the other legs have positive twists. Moreover, the first leg is incident to the same vertex as the $(n+1)$-st.
\item If $v \in F({\Gamma})$ (set of ``figure eights'' in the terminology of~\cite{EskMasZor}), then $v$ has exactly 1 incident edge. 
\item If $v \in P({\Gamma})$ (set of ``pairs of holes'' in the terminology of~\cite{EskMasZor}), then $v$ has exactly 2 incident edges. 
\item Each vertex in $R({\Gamma})$ has exactly two of the following half-edges: the leg $(n+1)$ or $(n+2)$, or the half-edge of a horizontal edge or of an edge to a vertex in~$P({\Gamma})$.
\item Each vertex in $R({\Gamma})$ has exactly one leg with index in $[\![1,n]\!]$. 
\end{itemize}
We say a chain~$\Gamma$ is \emph{odd} if all positive twists are odd. Denote by ${\rm CH}(\mu)$ and ${\rm CH}(\mu)^\oddd$  the set of (odd) chains compatible with $(m_1,\ldots,m_n,0,0)$. %(in which case the set $P(\Gamma)$ is empty).
\end{definition}
Observe that each vertex has an even number of incident half-edges with even twists. As $m(n+1)=m(n+2)=0$, it follows that each vertex in $R(\Gamma)$ has at least two incident half-edges with even twists. Hence, if $\Gamma$ is odd, then $P(\Gamma)$ is empty.

\begin{definition} Let $1\leq i\leq n$. A {\em  cylinder configuration marked by $x_1$}  is the data of: a cylinder configuration of a stratum of abelian differentials (in the sense of~\cite{EskMasZor} or~\cite{CheChe1}), and the choice of a side of one of the cylinders which is bounded by the marking $x_1$. 
\end{definition}

The data of a cylinder configurations marked by $x_1$ is equivalent to the choice of: (i)~a chain in ${\rm CH}(\mu)$,  (ii)~an order on the vertices in $F({\Gamma})$ connected to $v$ for all $v$ in $R({\Gamma})$, (iii)~an integer $1\leq a_e\leq m_e$ for all edges incident to a vertex in $F({\Gamma})$ or $P({\Gamma})$, (iv)~a choice of connected component of the space $\H(\mu(v))$ for all $v\in F(\Gamma)$ and $P(\Gamma)$. We use this fact to express the number $c_1(\mu)$ as a sum over chains. 

\begin{example}
On Figure~\ref{fig:chain}, we represented a chain for $\mu=(6,6,3,2,2,2,2)$. The vertices in $R(\Gamma), F(\Gamma),$ and $P(\Gamma)$ are of color green, red, and blue respectively.   Using the notation of~\cite{EskMasZor} (in particular the non-logarithmic convention), the configurations associated to this graph are of the form: 
\begin{eqnarray*}
\Rightarrow (\overline{\alpha_1+(1-\alpha_1)},1)\rightarrow (\overline{1},\overline{0},2,1)\Rightarrow (\overline{\alpha_2+(1-\alpha_2)},1)\rightarrow (\overline{0+0})\Rightarrow,\\
\text{or }\Rightarrow (\overline{\alpha_1+(1-\alpha_1)},1)\rightarrow (\overline{1},\overline{0},2,1)\Rightarrow (\overline{0+0})\rightarrow  (\overline{\alpha_2+(1-\alpha_2)},1) \Rightarrow,
\end{eqnarray*} for
$\alpha_1$ and $\alpha_2$ equal to 0 or 1. Here $\alpha_i=a_i-1$ in our system of notation. 
\begin{figure}  
\includegraphics[scale=0.3]{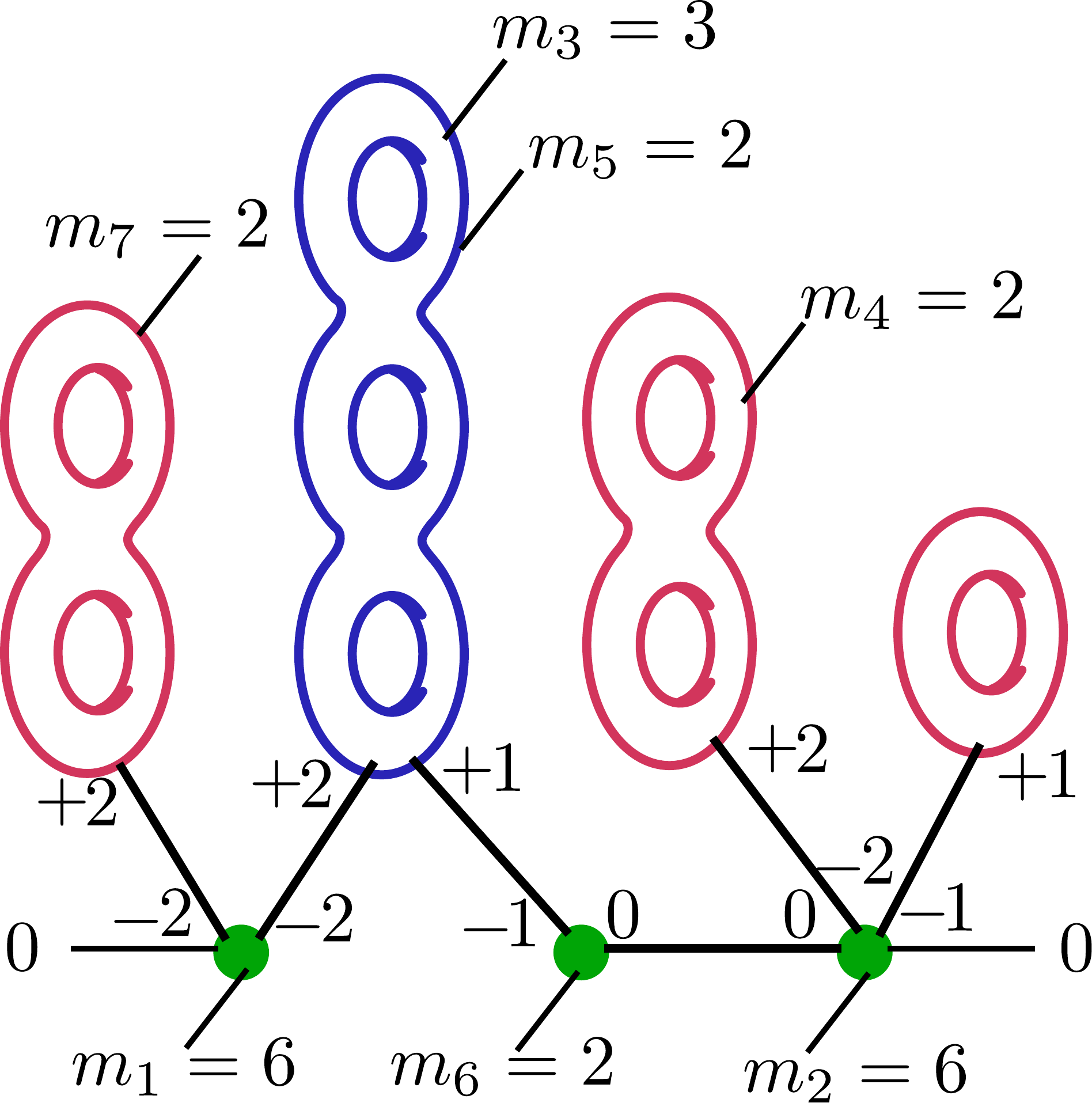}
\caption{Example of chain in ${\rm CH}(6,6,3,2,2,2,2)$.\label{fig:chain}}
\end{figure}
\end{example}
Let $\Gamma$ be a chain. We define the contribution of a vertex $v$ of $\Gamma$ 
%is defined 
according to the type of vertex:
\begin{itemize}[leftmargin=25pt]
\item If $v\in R(\Gamma)$, then $\widehat{c}(v) := m_i\,(n(v)-3)!,$ where $m_i$ is the marking of the unique leg incident to $v$;
\item If $v\in F(\Gamma)$, then $\widehat{c}(v):=m_e\,|\mu(v)|\, \v(\mu(v))$, where~$e$ is the unique edge incident to $v$;
\item If $v\in P(\Gamma)$, then $\widehat{c}(v):=|\mu(v)|\, \v(\mu(v)).$
\end{itemize}
Then, we define
$$
\widehat\cont(\Gamma) \defis \frac{1}{|{\rm Aut}(\Gamma)|}  \prod_{v\in V(\Gamma)} \widehat{c}(v).
$$
If $\mu$ is odd and $\Gamma$ is odd, then we define $\widehat{c^\spin}(v):=-m_i\,(n(v)-3)!$ for a vertex $v\in R(\Gamma)$, and $\widehat{c^\spin}(v)=-|\mu(v)|\, \v^\spin(\mu(v))$ for a vertex $v\in F(\Gamma)$, and we set:
$$
\widehat\cont^\spin(\Gamma)\defis\frac{1}{|{\rm Aut}(\Gamma)|}  \prod_{v\in V(\Gamma)} \widehat{c^\spin}(v).
$$

\begin{proposition}\label{pr:cchains}
The following identity holds 
\begin{eqnarray*}
c_1(\mu)&=& \frac{(-4\pi^2)^{-1}}{2\,m_1\, \v(\mu) }   \sum_{\Gamma \in {\rm CH}(\mu)} \widehat\cont(\Gamma).\\ 
\end{eqnarray*}
Moreover, if $\mu$ is odd, we have
\begin{eqnarray*}
c_1^\spin(\mu)&=& \frac{(-4\pi^2)^{-1}}{2\,m_1 \,\v(\mu) }   \sum_{\Gamma \in {\rm CH} (\mu)^\oddd} \widehat\cont^\spin(\Gamma).  
\end{eqnarray*}
\end{proposition}

\begin{proof}
Let ${\Gamma}$ be a chain graph. We choose a connected component $X_v$ of $\H(\mu(v))$ for all vertices in $F(\Gamma)$ and $P(\Gamma)$. All the marked cylinder configurations with chain graph ${\Gamma}$ and with the same choices of $(X_v)$ have the same Siegel--Veech constant, given by
$$
 \frac{(-4\pi^2)^{-1}}{m_1\cdot \v(\mu) } \left(\prod_{ v\in R({\Gamma})} \!\!\! m_i \right)  \times  \left( \prod_{v\in F({\Gamma})} \!\!\!   |\mu(v)|\cdot\v(X_v)\right) \times  \left( \prod_{v\in P({\Gamma})} \!\!\!  |\mu(v)|\cdot\v(X_v)\right),
$$
where, in the first product $m_i$ is the twist of the unique leg incident to $v\in R(\Gamma)$ (see~\cite{EskMasZor}, Formulas 13.1 and 14.4). Besides, there are
$$
\frac{1}{|{\rm Aut}(\Gamma)|} \left(\prod_{ v\in R({\Gamma})} \!\!\!  (n(v)-3)! \right)  \times  \left( \prod_{e\mapsto v\in F({\Gamma})} \!\!\!  m_e \right)
$$
such configurations. Indeed the first product accounts for all choices of orders on the vertices of $F(\Gamma)$ connected to the vertices in $R(\Gamma)$, while the second product accounts for the choice of the $a_e$ for edges incident to the vertices in $F(\Gamma)$. Thus the first identity follows as $c_1(\mu)$ is the sum of the Siegel--Veech constants of all configurations. 

To obtain the second identity, we recall that if $\Gamma\in {\rm CH}(\mu)\setminus {\rm CH}(\mu)^\oddd$ then half of the choices of tuples $(a_e)_{e\mapsto v\in F(\Gamma)}$ contribute to the even or odd component (see~\cite{EskMasZor}, Lemma~14.4). Thus the contribution of the odd and even components compensate and  $\Gamma$ contributes trivially to $c_1^\spin(\mu)$. Besides, if $\Gamma \in {\rm CH}(\mu)^\oddd$, then by \cite[Lemma~14.2]{EskMasZor} the parity of the configuration is the parity of
$$
\#R(\Gamma) + \sum_{e\mapsto v\in F(\Gamma)} a_e  + \sum_{v\in F(\Gamma)} \phi(X_v)
$$ 
where $\phi(X_v)$ equals $0$ or $1$ if $X_v$ is an even or odd component respectively. Thus, for each edge $e$ incident to a vertex in $F(\Gamma)$, we have $(a_e-1)/2$ terms which give the same parity while the other $(a_e+1)/2$ produce the inverse parity, thus only one of these choices of $a_e$ contributes. 
  \end{proof}

%%%%%%%%%%%%%%%%%%%%%%%
\subsection{Expanded chains} 
%%%%%%%%%%%%%%%%%%%%%%%

Proposition~\ref{pr:cchains} provides an expression of $c_1(\mu)$ as a sum over chains in ${\rm CH}(\mu)$, while Proposition~\ref{pr:dint} above provides an expression of $d_1(\mu)$ as a sum over backbone graphs. To compare $c_1$ and $d_1$, we introduce a family of sets of graphs of rational type ${\rm ECH}(\mu)_i$, the {\em expanded chains of complexity~$i$ } for $i=1,\ldots, n$ (see Definition~\ref{def:levelechain} below). We {will} also construct maps between the different sets of graphs of rational type constructed until here:
$$
\xymatrix{
{\rm BB}(\mu)_0\simeq {\rm ECH}(\mu)_1 \overset{F_2}{\leftarrow} {\rm ECH}(\mu)_2 \overset{F_3}{\leftarrow} \ldots \overset{F_{n}}{\leftarrow}  \hspace{-27pt} &{\rm ECH}(\mu)_{n} \ar[d]_F\\
&{\rm CH}(\mu).
}
$$
These maps will be used to compare the different expressions of $d_1(\mu)$ by applying Lemmas~\ref{lem:indf} and~\ref{lem:indg} repeatedly. 

\begin{definition}
 A {\em pre-expanded chain} is a graph of rational type with $n+2$ legs such that there exists a partition $R(\Gamma)=C({\Gamma}) \sqcup L({\Gamma})$ satisfying the following constraints:
\begin{itemize}[leftmargin=25pt]
\item $m(n+1)=m(n+2)=0$, and the other legs have positive twists. Moreover, the first leg is incident to the same vertex as the $(n+1)$-st.
\item Let $(v_0,\ldots, v_k)$ be the shortest path from the vertex~$v_0$ with the leg $n+1$ to the vertex~$v_k$ with leg $n+2$. A vertex is in $C(\Gamma)$ (the {\em core}) if and only if it appears in this path. Thus we have an ordering on the vertices of the core. 
\item A vertex $v$ in $C(\Gamma)$ is called a {\em bottom} or a {\em top} if for all $v'$ in $C(\Gamma)$ connected to $v$, we have $v\leq v'$ or $v> v'$ respectively.  
\item All half-edges with vanishing twists are incident to {bottoms} of $C(\Gamma)$. {In particular, all horizontal edges are between two bottoms.}
\item Each vertex in $D({\Gamma})$ has exactly one edge. 
\item Each vertex in $L(\Gamma)$ (the set of {\em links}), has exactly one edge to a lower vertex (it may have any number of edges to upper vertices).
\item If $v$ is a vertex in $L(\Gamma)$ or a vertex in $C(\Gamma)$ which is not a top, then $v$ has at least one leg in $[\![1,n]\!]$.
\end{itemize}
\end{definition}

\begin{definition} Let $\Gamma$ be an almost expanded chain. For all vertices of $\Gamma$, we denote by $\ind(v)$ the minimum of the indices of the legs incident to $v$ and $+\infty$ if there are no legs incident to $v$. Let $v$ be a top in $C(\Gamma)$. It determines a unique subpath of the core $C(\Gamma)$: 
$$(v_{k_1}, v_{k_1+1},\ldots, v_N=v,v_{N+1},\ldots,v_{k_2})$$
 such that $v_{N}$ is the unique top of the sequence, and $v_{k_1}$ and $v_{k_2}$ are the only bottoms. We say that $v$ is {\em admissible} if 
the minimum of $\ind(v_j)$ for $j=k_1+1,\ldots, k_2$ is reached for $j=N+1$. An almost expanded chain is an {\em expanded chain} if all tops of $C(\Gamma)$ are admissible.
\end{definition}

\begin{definition}\label{def:levelechain} If $1\leq i\leq n$, then we denote by ${\rm ECH}(\mu)_i$ the set of \emph{expanded chains of complexity $i$}, i.e.\ the chains satisfying:
\begin{itemize}[leftmargin=25pt]
\item  for all $1\leq j\leq i$, we have: if the $j$-th leg is incident to a vertex $v$, then $v$ is not a top of $C(\Gamma)$ and  either $v$ in $D(\Gamma)$ or $j=\ind(v)$.
\item for all $v$ in $R(\Gamma)$, we have $\ind(v)\leq i$ or $v$ is a top of $C(\Gamma)$.
\end{itemize}
\end{definition}

To compare the different sets of graphs we define the functions $F_2,\ldots,F_{n}$ and $F$ (see the above diagram).

\subsubsection{{ Construction of the maps~$F_i$.}} Let $\Gamma$ be a graph in ${\rm ECH}(\mu)_i$. We construct the image of $\Gamma$ as follows:
\begin{itemize}[leftmargin=25pt]
\item If the $i$-th leg is incident to a vertex of $D(\Gamma)$, then $F_i(\Gamma)=\Gamma$.
\item If the $i$-th leg is incident to a vertex of $L(\Gamma)$, then $F_i(\Gamma)$ is obtained by contracting the unique edge to a lower vertex.
\item If the $i$-th leg is on a vertex $v_j$ of $C(\Gamma)$ and $v_{j-1}$ is not a top, then $F_i(\Gamma)$ is obtained by contracting the edge between $v_{j-1}$ and $v_{j}$; if $v_{j-1}$ is a top, we also contract the edge between $v_{j-2}$ and $v_{j-1}$.
\end{itemize}
{Note that $F_i(\Gamma)$ then satisfies the first condition of elements in ${\rm ECH}_{i-1}$ by the admissibility condition, and the second condition because if there is a vertex $v$ in $\Gamma$ with $\ind(v)=i$, then this vertex is merged with a vertex which has a leg of smaller index incident to it.}

\subsubsection{{ Construction of the map~$F$.}} Let $\Gamma$ be a graph in ${\rm ECH}(\mu)_n$. The image of $\Gamma$ is defined by contracting all edges which are not incident to at least one bottom vertex of $C(\Gamma)$, {where the genus of a merged vertex is the sum of the genera of the previous vertices}.

\begin{example}
In Figure~\cref{fig:echain}, we represented an example of an expanded chain in ${\rm ECH}(6,6,3,2,2,2,2)_i$ for $i=6$ or $7$ (for simplicity we did not put the twist at the edges as they may be computed from the twists at the legs). The vertices of $C(\Gamma)$ are black dots while the vertex in $L(\Gamma)$ is a white dot. Remark that the admissibility condition is satisfied, as the marking on the vertex following the unique top of $C(\Gamma)$ has index 3 which is smaller than 5 and 6.

Besides, this expanded chain is mapped to the chain of Example~\ref{fig:chain} under $F$.  We have surrounded in green, red, or blue the subgraphs that have to be contracted to obtain this chain. 
\begin{figure}  
\includegraphics[scale=0.3]{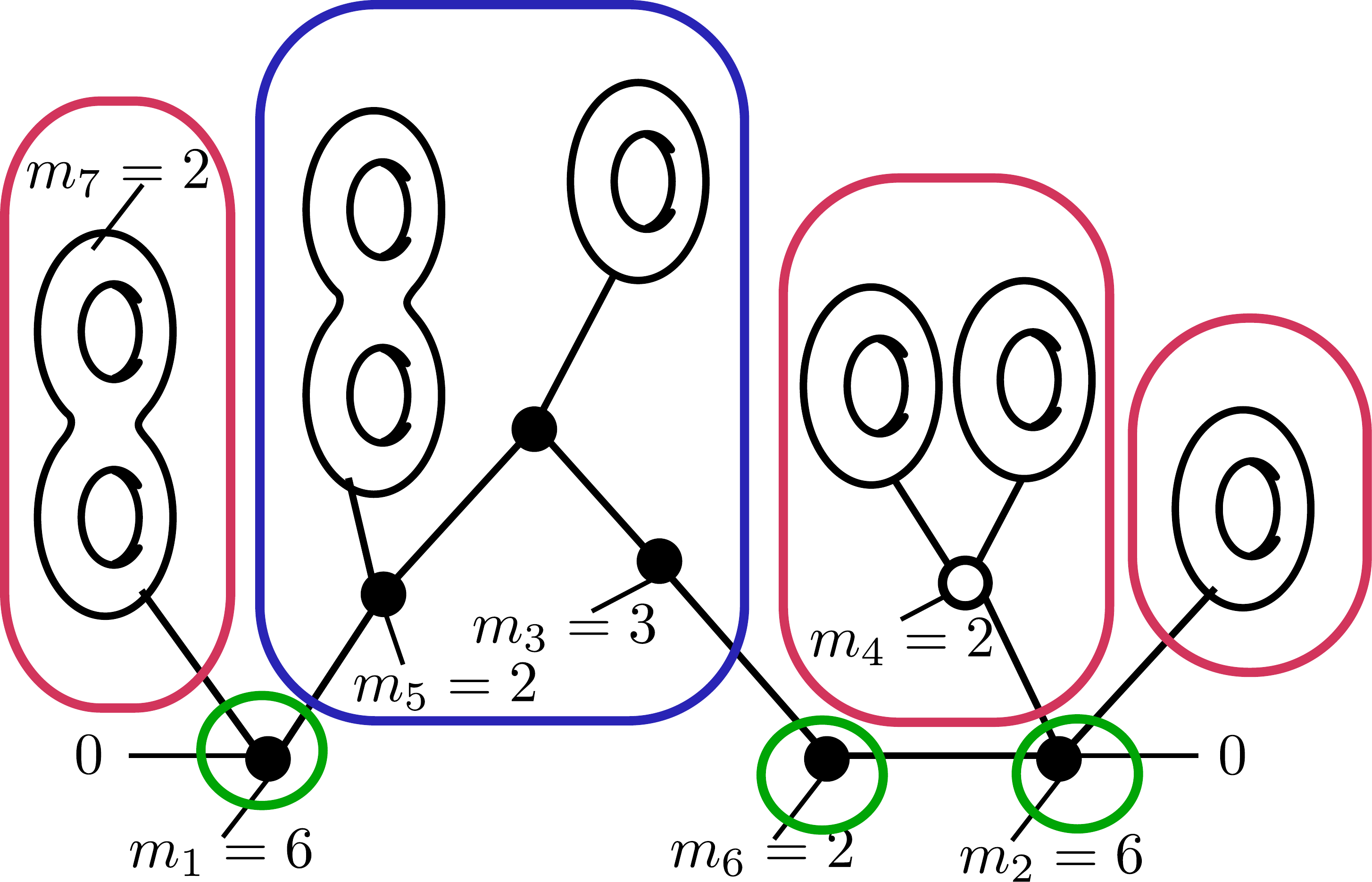}
\caption{An expanded chain in ${\rm ECH}(6,6,3,2,2,2,2)_i$ for $i=6$ or $7$.\label{fig:echain}}
\end{figure}
\end{example}

\subsubsection{Odd expanded chains} If $\mu$ is odd, then an expanded chain $\Gamma$ is {\em odd} if the all positive twists are odd. In particular, $C(\Gamma)$ contains only bottom vertices. Indeed, each vertex has an even number of incident half-edges with even twists, thus each vertex of the core has exactly 2 incident half-edges with even twists (the ones connecting to the previous and next vertex or the legs $n+1$ and $n+2$). As all of these twists are equal to 0, all vertices are bottom and all edges of the core are horizontal. The functions $F_i:{\rm ECH}^\oddd_{i}(\mu)\to {\rm ECH}^\oddd_{i-1}(\mu)$ and $F: {\rm ECH}^\oddd_{n}(\mu)\to {\rm CH}^\oddd(\mu)$ are defined by restricting the functions $F$ and $F_i$ to the sets of odd expanded chains.

%%%%%%%%%%%%%%%%%
\subsection{Contribution of expanded chains} 
%%%%%%%%%%%%%%%%%
Let $\Gamma$ be an expanded chain. The contribution~$c(v)$ of a vertex~$v$ of $\Gamma$ is defined according to the type of vertex:
\begin{itemize}[leftmargin=25pt]
\item If $v\in D(\Gamma)$, then $c(v):=m_e\, \v(\mu(v))$, where $m_e$ is the twist at the edge incident to $v$ and $\mu(v)$ is the vector of twists of half-edges incident to $v$.
\item If $v \in L(\Gamma)$, or if $v\in C(\Gamma)$ is neither a top nor a bottom, then we set $c(v):=m_{j}\, f(\mu(v))$, where $j=\ind(v)$ and $\mu(v) = (m_j,m_e,\ldots)$ with $e$ the unique edge to a lower vertex.
% and the first entries of $\mu$ are $m_{j}$ and $m_e$  the twist at the unique edge to a lower vertex.
\item If $v$ is a top of $C(\Gamma)$, then $c(v):=(m_e+m_{e'})\, f(\mu(v))$ where $m_e$ and $m_{e'}$ are the twists at the two edges~$e$ and $e'$ to lower vertices, and $\mu(v) = (m_e,m_{e'},\ldots)$. %are the first entries of $\mu(v)$.
\item If $v$ is a bottom of $C(\Gamma)$, then $c(v):=m_j \,\varphi(\mu(v))$, where $j=\ind(v)$, and $\mu(v)=(m_j,\ldots,m_e,m_{e'})$, where $e$ and $e'$ are the edges to the previous and next vertex in $C(\Gamma)$ if there is one, and else, $m_e=0$ or $m_{e'}=0$ respectively. 
%the first entry is $m_j$ while the last entries are $m_e$ and $m_{e'}$, the twists at edges to the previous and next vertex in $C(\Gamma)$ if there is one and 0 otherwise.
\end{itemize}
Then, %with this notation 
we define
$$
\cont(\Gamma)\=\frac{1}{|{\rm Aut}(\Gamma)|}  \prod_{v\in V(\Gamma)} c(v).
$$
Moreover, if $\mu$ is odd and $\Gamma$ is an odd expanded chain, and $v$ is a vertex of $\Gamma$, then we define $c^\spin(v)$ and $\cont^\spin(\Gamma)$ %as $\cont(\Gamma)$ 
by replacing the function $\v$ by $\v^{\spin}$ and the function $\varphi$ by $\varphi^\spin$ in the definition of $c(v)$ and $\cont(\Gamma)$.
\begin{proposition}\label{pr:echech}
Let $1< i\leq n$ and let $\Gamma\in {\rm ECH}(\mu)_{i-1}$. We have
$$\cont(\Gamma)\=\sum_{\Gamma' \in F_{i}^{-1}\{\Gamma\}}  \cont(\Gamma').$$
Besides, if $\Gamma$ is odd, then
$$
\cont^\spin(\Gamma)\=\sum_{\Gamma' \in F_{i}^{-1}\{\Gamma\}}  \cont^\spin(\Gamma'),
$$
where in the last sum we restrict $F_i$ to the sets of \emph{odd} expanded chains.
\end{proposition}

\begin{proof}
Let $i>1$ and $\Gamma\in {\rm ECH}(\mu)_{i-1}$. Any graph in $F_i^{-1}\{\Gamma\}$ is obtained by modifying the vertex~$v$ carrying the leg $i$ and not the others. Thus we show that the proposition holds by studying each possible type of vertex for $v$. 

If $v\in D(\Gamma)$, then $\Gamma\in {\rm ECH}(\mu)_{i}$ 
%and it is the only graph mapped to himself under $F_{i}.$ 
and $F_i^{-1}\{\Gamma\}=\{\Gamma\}.$
Thus, the proposition holds trivially.

If $v$ is in $L(\Gamma)$, then $c(v)$ is given in terms of the function~$f$.  Using Lemma~\ref{lem:indf}, we may write this function as a sum over all twisted graphs of type~$f_2$, i.e.\ as all possible ways to split $v$ into 2 vertices (leaving $\ind(v)$ and the edge towards a lower vertex together on the lowest of the created vertices while $i$ is on the upper one). All the graphs of $F_i^{-1}\{\Gamma\}$ are obtained in this way, and we may check that the induction formula of Lemma~\ref{lem:indf} gives the right contribution of any element of $F_i^{-1}\{\Gamma\}$. 

Therefore from now on, we assume that $v$ is in $C(\Gamma)$. First, let us remark that $v$ cannot be a top of $C(\Gamma)$. Namely, $v$ cannot be the result of contracting the edge between a top $v_j$ and a vertex $v_{j-1}$, as that would violate the first condition of expanded chains of complexity~$i$. Moreover, if $v$ is obtained by contracting the edge between a vertex~$v_j$, a top $v_{j-1}$ and a vertex $v_{j-2}$, by the admissibility condition, we know that $v$ is a bottom (and not a top).

If $v$ is neither a top nor a bottom, then its contribution is given by the function~$f$. Then we apply the recursion of Lemma~\ref{lem:indf} and we write $f$ as a sum over graphs of type~$f_h$ where $h$ is the half-edge of the edge that connects $v$ to the previous vertex of $C(\Gamma)$. Either we obtain two vertices in $C(\Gamma)$, where the closest to $v_0$ carries the leg~$\ind(v)$ while the other carries the leg $i$; or a vertex in $C(\Gamma)$ with the leg $\ind(v)$ and a vertex in $L(\Gamma)$ carrying the leg $i$. All the graphs of $F_i^{-1}\{\Gamma\}$ are obtained in this way, and we may check that the induction formula of Lemma~\ref{lem:indf} gives the right contribution of any element of $F_i^{-1}\{\Gamma\}$.

If $v$ is a bottom in $C(\Gamma)$, then its contribution is given by the function~$\varphi$. Any graph in $F_i^{-1}\{\Gamma\}$ is obtained by splitting $v$ into 2 or 3 vertices of type~$\varphi_1, \varphi_1', \varphi_2,$ or $\varphi_3$. Then the induction formula of Lemma~\ref{lem:indg} shows that the contribution of $\Gamma$ may be written as the sum of the contributions of $F_i^{-1}\{\Gamma\}$ as in the previous case. \end{proof}

%%%%%%%%%%%%%%%%%
\subsection{Contribution of chains}
%%%%%%%%%%%%%%%%% 
The purpose of this section is to show the following identities.
\begin{proposition}\label{pr:echch} 
Let $\Gamma\in {\rm CH}(\mu)$. We have: 
$$\widehat\cont(\Gamma)=\sum_{\Gamma' \in F^{-1}\{\Gamma\}}  \cont(\Gamma'), $$
and if $\Gamma$ is odd, then
$$
\widehat\cont^\spin(\Gamma)=\sum_{\Gamma' \in F^{-1}\{\Gamma\}}  \cont^\spin(\Gamma').
$$
\end{proposition}

To prove these identities we need two extra families of combinatorial objects, called rooted trees and expanded pairs of holes. 
\begin{definition} Let $\Sigma \subset [\![1,n]\!]$ and $p$ be a positive integer. A {\em rooted tree} is a graph of rational type with $1+|\Sigma|$ legs {indexed by $\{r\}\cup \Sigma$}, which is either the trivial graph (i.e., the stable graph with one vertex and no edges) or is such that:
\begin{itemize}[leftmargin=25pt]
\item No edge is horizontal.
\item Each vertex in $D({\Gamma})$ has exactly one incident edge.
\item {The vertex with the leg~$r$ is called the \emph{root}. The root has $2$ legs and no edges to lower vertices. All other vertices of $R(\Gamma)$ have $1$ leg and $1$ edge to a lower vertex (they may have any number of edges to upper vertices)}.
%Each vertex of $R(\Gamma)$ has either $1$ leg and an edge to a lower vertex, or $2$ legs including the first one and no edges to lower vertices.
\end{itemize}
We denote by ${\rm RT}(\mu,\Sigma,p)$ the rooted trees compatible with $(p)+(m_i)_{i\in \Sigma}$.
\end{definition}

\begin{definition} Let $\Sigma\subset [\![2,n]\!]$, $p_1, p_2$ be positive integers and $I\in \mathbb{R} \setminus \Sigma $.
A {\em pre-expanded pair of holes} is a graph of rational type~$\Gamma$ with legs indexed by the set $\Sigma\cup \{h_1,h_2\}$, and with a partition $R(\Gamma)=C(\Gamma)\sqcup L(\Gamma)$   satisfying:
\begin{itemize}[leftmargin=25pt]
\item All legs have positive twists and there are no horizontal edges.
\item Let $(v_0,\ldots, v_k)$ be the shortest path from the vertex with the leg $h_1$ to the vertex with the leg $h_2$. A vertex is in $C(\Gamma)$ if and only if it appears in this path. 
\item  There is exactly one top in $C(\Gamma)$, i.e.\ there is precisely one vertex $v$ whose edges to other vertices of $v'\in C(\Gamma)$ satisfy $v>v'$. No leg in $\Sigma$ is incident to the top.
\item Each vertex in $D({\Gamma})$ has exactly one edge. 
\item Each vertex in $L(\Gamma)$ (the set of {\em links}), has exactly one edge to a lower vertex (it may have any number of edges to upper vertices).
\item If $v$ is a vertex in $L(\Gamma)$ or a vertex in $C(\Gamma)$ which is not the top, then $v$ has exactly one leg in $\Sigma$.
\end{itemize}
It is an {\em expanded pair of holes} in ${\rm EP}(\mu,\Sigma,p_1,p_2,I)$ if it is compatible with $(m_i)_{i\in \Sigma}+(p_1,p_2)$, and the following \emph{admissibility condition} holds: 
\begin{itemize}[leftmargin=25pt]
\item Either the leg {$h_2$} is incident to the top vertex (i.e.\ the top is the last vertex of the core), and $I$ is smaller than all legs incident to vertices in $C(\Gamma)$;
\item Or the smallest index~$j$ of a leg incident to a vertex of $C(\Gamma)$ is smaller than $I$ and is incident to the vertex directly after the top.
\end{itemize}
\end{definition} 

\begin{example}
In Figure~\ref{fig:echain}, examples of rooted trees or expanded pairs of holes are given by the subgraphs surrounded by red or blue lines respectively.
\end{example}

If $\Gamma$ is a rooted tree or an expanded pair of holes, and $v$ is a vertex of $\Gamma$, then the contribution~$c(v)$ of~$v$ is defined as for expanded chains. Here, for rooted trees, we set $R(\Gamma)=L(\Gamma)$ and identify $m_e$ with $p$ for the root. Also, for expanded pair of holes, we identify $p_1$ and $p_2$ with the twists $m_e$ and $m_{e'}$ at the first and last vertices of the core respectively.
 Then we define $\cont(\Gamma)$ as $|{\rm Aut}(\Gamma)|^{-1} \prod_v c(v)$.  If $\mu$ is odd and $\Gamma$ is a rooted tree then we define $\cont^\spin(\Gamma)$ in the same way.

\begin{lemma}[{\cite[Section~3.5 and 6.1]{CheMoeSauZag}}]  \label{lem:rt}
For all positive integers $p$ and $\Sigma \subset [\![2,n]\!]$, we have
$$
\left(p+ \sum_{i\in \Sigma} m_i\right)  \v\left((p)+(m_i)_{i\in \Sigma}\right) \= \!\!\!\!\!\!\!  \sum_{\Gamma \in {\rm RT}(\mu,\Sigma,p)}\!\!  \cont(\Gamma),
$$
and if $\mu$ is odd, then we have:
$$
 \left(p+ \sum_{i\in \Sigma} m_i\right) \v^\spin\left((p)+(m_i)_{i\in \Sigma}\right)\=\!\!\!\!\!\!\! \sum_{\Gamma \in {\rm RT}(\mu,\Sigma,p)} \!\! \cont^\spin(\Gamma).
$$
\end{lemma}
We will show the following analogue result for expanded pair of holes.
\begin{lemma}\label{lem:eph}
For all positive integers $p_1,p_2$, $\Sigma\subset [\![2,n]\!]$ and $I\in \RR\setminus \Sigma$,  we have:
$$
\left(p_1+p_2+ \sum_{i\in \Sigma} m_i\right)  \v\left((m_i)_{i\in \Sigma}+(p_1,p_2)\right)=   \!\!\!\!\!\!\!  \sum_{\Gamma \in {\rm EP}(\mu,\Sigma,p_1,p_2,I)} \!\! \cont(\Gamma).
$$
\end{lemma}

\begin{proof}
We fix a choice of $\mu, \Sigma, p_1, p_2$. We denote by $S(I)$ the sum on the right-hand side of the identity of the Lemma. This function is locally constant on $\RR\setminus \Sigma$, thus to prove the lemma we proceed in two steps: first, we show that it is valid when $I=0$; then we show that  $S(i-1/2)-S(i+1/2)=0$  for all $i\in \Sigma$.\bigskip

\step{The case $I=0$.} We proceed by induction on the size of $\Sigma$. If $\Sigma$ is empty, then a graph in ${\rm EP}(\mu,\Sigma,p_1,p_2,0)$ is a backbone s.t.\ the legs $h_1$ and $h_2$ are incident to $v_{-1}$. Thus, we may apply Proposition~\ref{pr:volint} for $\mu=(p_1,p_2)$, which implies the lemma in this case.

If $\Sigma$ is non-empty, then the admissibility condition implies that the marking $h_2$ is incident to the top.  Then a graph in ${\rm EP}(\mu,\Sigma,p_1,p_2,0)$ is determined by: \begin{enumerate}
 \item  an element $i\in \Sigma\cup\{h_2\}$;
 \item  an integer $k\geq 1$; 
 \item a partition $\Sigma\setminus \{i\}= \Sigma_1\sqcup \ldots \sqcup \Sigma_k$, and a partition $p_1+m_i-k=p_1'+\ldots+p_k'$ (where $m_i=p_2$ if $i=h_2$)
 \item   a rooted tree in ${\rm RT}(\mu,\Sigma_j,p_j')$ for $1\leq j\leq k$ if $h_2\notin \Sigma_j$, or an element of ${\rm EP}(\mu,\Sigma_j,p_j',p_2,0)$ otherwise.
 \end{enumerate}
 Indeed, with this datum we construct a graph in ${\rm EP}(\mu,\Sigma,p_1,p_2,0)$ by attaching the $k$ graphs of the last part of the data to a vertex of genus $0$ with the markings $h_1$ and $i$. {Here, we replace an half-edges with marking $p_j'$ by an edge~$e_j$ with twist $m_{e_j}=p_j'$ to this vertex with markings $h_1$ and $i$.}
Using this fact, we may rewrite the sum defining $S(0)$ as follows:
 \begin{eqnarray*}
 \sum_{\substack{ i\in \Sigma\\ k\geq 1}}  \!\!\!\!\!\!   & & \!\!\! \sum_{\substack{ \Sigma_1\sqcup \ldots \sqcup \Sigma_k=(\Sigma\setminus\{i\}) \cup\{h_2\}\\ p_1'+\ldots+p_k'=p_1+m_i-k}} \!\!\! \frac{m_i \, f(p_1,m_i,-p_1',\ldots,-p_k')}{(k-1)!} \\
&&\qquad\qquad \times \left(\sum_{\Gamma\in {\rm EP}(\mu,\Sigma_1,p_1',p_2,0)} \!\!\!\!\!\!\!\!\!\!\!\!\cont(\Gamma)\right) \times\ \prod_{j=2}^k \left(\sum_{\Gamma\in {\rm RT}(\mu,\Sigma_j,p_j')} \!\!\!\!\!\!\!\!\! \cont(\Gamma)\right)\\
  \+\sum_{\substack{  k\geq 1}}  \!\!\!\!\!\! && \!\!\! \sum_{\substack{ \Sigma_1\sqcup \ldots \sqcup \Sigma_k=\Sigma\\ p_1'+\ldots+p_k'=p_1+p_2-k}} \!\!\! \frac{(p_1+p_2) \, f(p_1,p_2,-p_1',\ldots,-p_k')}{k!}\, \prod_{j=1}^k \left(\sum_{\Gamma\in {\rm RT}(\mu,\Sigma_j,p_j')} \!\!\!\!\!\!\!\!\! \cont(\Gamma)\right)\!\!\raisebox{-10pt}{.}
 \end{eqnarray*}
 The first sum accounts for the contribution of graphs where $h_2$ is not incident to the same vertex as $h_1$ (thus one of the descendants of the main vertex of genus~$0$ is distinguished as it carries the leg $h_2$) while the second sum accounts for the contribution of graphs with $h_1$ and $h_2$ incident to the same vertex.
  
 Therefore we may compute the sum $S(0)$ recursively by applying Lemma~\ref{lem:rt} and the induction hypothesis to obtain that $S(0)$ is given by 
  \begin{eqnarray*}
 \sum_{\substack{ i\in \Sigma\\ k\geq 1}}  & & \!\!\!\!\!\!\!\!\! \sum_{\substack{ \Sigma_1\sqcup \ldots \sqcup \Sigma_k=(\Sigma\setminus\{i\}) \cup\{h_2\} \\ p_1'+\ldots+p_k'=p_1+m_i-k}} \!\!\! \frac{m_i \, f(p_1,m_i,-p_1',\ldots,-p_k')}{k!} \\
&&\times \, \prod_{j=1}^k \left(p_j'+\sum_{i'\in \Sigma_j} m_{i'}\right) \v\bigl((p_j')+(m_{i'})_{i'\in \Sigma_j}\bigr) \\
  \+ \sum_{k\geq 1}  \!\!\! \!\!\! \!\!\! &&\sum_{\substack{ \Sigma_1\sqcup \ldots \sqcup \Sigma_k=\Sigma\\ p_1'+\ldots+p_k'=p_1+p_2-k}} \!\!\! \frac{(p_1+p_2) \, f(p_1,p_2,-p_1',\ldots,-p_k')}{k!} \\
  &&\times \prod_{j=1}^k \left(p_j'+\sum_{i'\in \Sigma_j} m_{i'}\right) \v\bigl((p_j')+(m_{i'})_{i'\in \Sigma_j}\bigr),
 \end{eqnarray*}
where, again, $m_{h_2}=p_2$. We apply Proposition~3.11 of~\cite{CheMoeSauZag} to deduce that:
$$
S(0)\=\sum_{i\in \Sigma}\left(m_i \, \v(\Sigma+(p_1,p_2))\right) \+ (p_1+p_2)\,\v(\Sigma+(p_1,p_2)),
$$
which is the desired identity.
\bigskip

\step{Crossing an element in $\Sigma$.} We fix $i$ in $\Sigma$. We first remark that the contribution of an element in the sum defining $S(I)$ does not depend on $I$. Indeed, the dependence of $S(I)$ on $I$ is uniquely given by the set of graphs that contribute. Thus we need to determine which graphs contribute to $S(i+1/2)$ but not to $S(i-1/2)$ and conversely. 

{We assume that $i=\min\{\ind(v) \mid v \in C(\Gamma)\}$. Namely, if this is not the case then the same graphs contribute to $S(i+1/2)$ and $S(i-1/2)$. Now,} A graph contributes to $S(i+1/2)$ but not to $S(i-1/2)$ if and only if the label $i$ belongs to $C(\Gamma)$ and $h_2$ is incident to the top. We denote by $S^+$ the sum of the contributions for graphs of this type. Conversely, a graph contributes to $S(i-1/2)$ but not to $S(i+1/2)$ if and only if the label $i$ is incident to the vertex following the top in $C(\Gamma)$. We denote by $S^-$ the sum of the contributions for graphs of this type. We will show that $S^+=S^-$ to finish the proof of the lemma.

To do so, we introduce a family of sets of pre-expanded pairs of holes $S(i,\ell)$ for all $\ell\geq 0$. A pre-expanded pair of holes $\Gamma{\in \mathrm{EP}(\mu,\Sigma,p_1,p_2,I)}$ {compatible with $(m_i)_{i\in \Sigma}+(p_1,p_2)$} belongs to $S(i,\ell)$ if:
\begin{itemize}[leftmargin=25pt]
\item the leg $i$ is incident to a vertex of $C(\Gamma)$ and if a leg $i'\in \Sigma$ is incident to a vertex of $C(\Gamma)$ then $i'>i$.
\item if the top vertex is the $t$-th vertex of the core, and the vertex carrying $i$ the $s$-th of the core then either: $\ell>0$ and the top is the last vertex of the core, {\em OR} we have $t-s+1=\ell$.
\end{itemize}
In particular with this notation we have $S(i,0)=S^-$, while $S(i,\ell)=S^+$ if $\ell$ is sufficiently large. Then we show that:
$$
\sum_{\Gamma \in S(i,\ell)} \cont(\Gamma)\=\sum_{\Gamma \in S(i,\ell+1)} \cont(\Gamma)
$$
for all $\ell\geq 0$.  Indeed, for each graph $\Gamma\in S(i,\ell)$: either the top vertex is the last vertex and then it belongs to $S(i,\ell+1)$ too, or we apply Lemma~\ref{lem:exch0} to the subgraph made of the top vertex and the next vertex in the $C(\Gamma)$. Then  Lemma~\ref{lem:exch0} exchange the roles of these two vertices and thus the difference between the positions of the top vertex and the vertex carrying the $i$-th leg is augmented by $1$. We thus obtain a summation on graphs in $S(i,\ell+1)$. \end{proof}

\begin{proof}[End of the proof of Proposition~\ref{pr:echch}.]  To finish the proof, we simply remark that the datum of an expanded chain is equivalent to the datum of: (i) a chain $\Gamma$; (ii) a rooted tree for each vertex of $F(\Gamma)$; (iii) an expanded pair of holes for each vertex of $P(\Gamma)$. Moreover, applying Lemmas~\ref{lem:rt} and~\ref{lem:eph} provides the equality between the contribution of a vertex in $F(\Gamma)$ or $P(\Gamma)$ and the sum of the contributions of the rooted trees or expanded pairs of holes associated to this vertex. {Note that the factor $m_e$ in the contribution $m_e |\mu(v)| \v(\mu(v))$ of a vertex $v\in F(\Gamma)$ does not occur in Lemma~\ref{lem:rt}, but does occur in the function~$\varphi(\mu(v'))$ in the contribution of $v'$, where $v'$ is the vertex of the core of the chain connected to $v$.}
\end{proof}

\subsection{End of the proof of Theorem~\ref{th:mainint}}  An expanded chain in ${\rm ECH}(\mu)_1$ has only one vertex in $R(\Gamma)$, and thus uniquely determines a backbone graph in  ${\rm BB}(\mu)_0$. Moreover, Proposition~\ref{pr:dint} may be rewritten as:
$$
 d_1(\mu) \= \frac{1}{2 m_1}\sum_{\Gamma\in {\rm ECH}(\mu)_1} \cont(\Gamma).
$$
Thus using Propositions~\ref{pr:echech},~\ref{pr:echch} and~\ref{pr:cchains} successively we obtain:
\begin{eqnarray*}
d_1(\mu)&=& \frac{1}{2 m_1}  \sum_{\Gamma\in {\rm ECH}(\mu)_n} \cont(\Gamma)\\
&=& \frac{1}{2 m_1}  \sum_{\Gamma\in {\rm CH}(\mu)} \widehat\cont(\Gamma)\\
&=&  (-4\pi^2) \, c_1(\mu)\cdot \v(\mu),
\end{eqnarray*}
which is the first identity stated in Theorem~\ref{th:mainint}. If $\mu$ is odd, then the second statement of Theorem~\ref{th:mainint} is obtained similarly by applying the spin counterpart of Propositions~\ref{pr:dint},~\ref{pr:echech},~\ref{pr:echch} and~\ref{pr:cchains}.

\appendix
\section{Character tables of the spin symmetric group and the Sergeev group for \texorpdfstring{$d\leq 5$}{d <= 5}}
For $d\leq 5$, and $G$ being one the groups $\Spin_d,\ASpin_d,\Sergeev_d$ and $\Sergeev^0_d$, we compute the character table of all irreducible \emph{spin} representations, i.e., we assume the central element $\varepsilon\in G$ acts by~$-1$. If the characters are $\chi_1,\chi_2,\ldots$ and the conjugacy classes $\mathcal{C}_1,\mathcal{C}_2,\ldots$, we write
\[\begin{array}[t]{c | c c}\hline
G & \mathcal{C}_1 & \cdots \\
|G| & |\mathcal{C}_1| & \cdots \\ \hline
 \chi_1 & \chi_1(\mathcal{C}_1) & \cdots \\
\vdots & \vdots  &  \\\hline
\end{array}\vspace{5pt}\]
for the character table of $G$. Note that $\chi_i(\epsilon\mathcal{C}_j)=-\chi_i(\mathcal{C}_j)$. Thus, we omit all conjugacy classes for which $\mathcal{C}_j\cap \epsilon \mathcal{C}_j\neq \emptyset$, and else pick only one of the two conjugacy classes $\mathcal{C}_j$ and $\varepsilon \mathcal{C}_j$.  Row and column orthogonality relations are satisfied, i.e.
\begin{align*}
& \sum_{i} |\mathcal{C}_j|\,\chi_i(\mathcal{C}_j)\,\overline{\chi_i(\mathcal{C}_{j'})} \= \frac{1}{2}|G|\delta_{j,j'} \text{ for all } j,j'\\
 &\sum_{j} |\mathcal{C}_j|\,\chi_i(\mathcal{C}_j)\,\overline{\chi_{i'}(\mathcal{C}_{j})} \= \frac{1}{2}|G|\delta_{i,i'} \text{ for all } i,i'.
\end{align*}
The first factor $\frac{1}{2}$ appears because we omit half of the conjugacy classes, as explained before. The second factor $\frac{1}{2}$ appears because we omit the non-spin representations, which correspond to representations of the quotient $G/\epsilon$.\medskip

\begin{longtable}[l]{l l l l}
%%%%%%%%%%%%%%%%%%%
%1
%%%%%%%%%%%%%%%%%%%
\step{${d=1}$}\\\nopagebreak
$\begin{array}[t]{c | c}\hline
\Spin_1 & e \\
2 & 1 \\ \hline
 1 & 1 \\\hline
\end{array}$ & \quad
$\begin{array}[t]{c}\\[-4	pt] \ASpin_1 \simeq \Spin_1\end{array}$ &\quad
%\\[-2pt]
$\begin{array}[t]{c | c c} \hline
\Sergeev_1 & e & C_{e,1} \\
4 & 1 & 1  \\\hline
1+ & 1 & +\ii \\ 
 1- & 1  & -\ii \\\hline
\end{array}$ &\quad
$\begin{array}[t]{c | c c} \hline
\Sergeev_1^0 & e \\
2 & 1 \\ \hline
1 & 1\\ \hline
\end{array}$
\end{longtable}
\begin{longtable}[l]{l l}
%\\[65pt]
%%%%%%%%%%%%%%%%%%%
%2
%%%%%%%%%%%%%%%%%%%
\step{${d=2}$}\\\nopagebreak
$\begin{array}[t]{c | c c} \hline
\Spin_2 & e & (12) \\
4 & 1 & 1 \\\hline
2+ & 1 & +\ii \\ 
 2- & 1  & -\ii \\\hline
\end{array}$&
$\begin{array}[t]{c | c c} \hline
\ASpin_2 & e \\
2 & 1 \\ \hline
2 & 1\\ \hline
\end{array}$ 
\\[-2pt]
$\begin{array}[t]{c | c c} \hline
 \Sergeev_2 & e & C_{(12),1} \\
16 & 1 & 2 \\\hline
2+ & 2 & +\ii\sqrt{2} \\ 
2- & 2  & -\ii\sqrt{2} \\\hline
\end{array}$&
$\begin{array}[t]{c | c c} \hline
 \Sergeev_2^0 & e \\
 8 & 1 \\ \hline
2 & 2\\ \hline
\end{array}$
\\[65pt]
%%%%%%%%%%%%%%%%%%%
%3
%%%%%%%%%%%%%%%%%%%
\step{${d=3}$}\\\nopagebreak
$\begin{array}[t]{c | c c c} \hline
\Spin_3 & e & (123) & (12) \\
12 & 1 & 2 & 3 \\\hline
3 & 2 & 1 & 0 \\ 
2,1+ & 1 & -1 & +\ii \\ 
2,1- & 1 & -1 & -\ii \\\hline
\end{array}$&
$\begin{array}[t]{c | c c c} \hline
\ASpin_3 & e & (123) & (321) \\
6 & 1 & 1 & 1 \\\hline
3+ & 1 & \frac{1}{2}+\frac{1}{2}\ii\sqrt{3} & \frac{1}{2}+\frac{1}{2}\ii\sqrt{3} \\ 
3- & 1 & \frac{1}{2}-\frac{1}{2}\ii\sqrt{3} & \frac{1}{2}-\frac{1}{2}\ii\sqrt{3} \\ 
2,1 & 1 & -1 & -1 \\\hline
\end{array}$ 
\\[-2pt]
$\begin{array}[t]{c | c c c} \hline
\Sergeev_3 & e & C_{(123)} & C_{(123),1}\\
 96 & 1 & 8 & 8 \\\hline
3+ & 4  & 1 & +\ii\sqrt{3} \\ 
3- & 4  & 1 & -\ii\sqrt{3} \\
2,1 & 4 & -2 & 0 \\\hline
\end{array}$&
$\begin{array}[t]{c | c c c } \hline
\Sergeev_3^0 & e & C_{(123)} & C_{(12)}\\
48 & 1 & 8 & 6 \\ \hline
3 & 4 & 1 & 0 \\
2,1+ & 2 & -1 & +\ii\sqrt{2} \\
2,1- & 2 & -1 & -\ii\sqrt{2}\\ \hline
\end{array}$
\\[75pt]
%%%%%%%%%%%%%%%%%%%
%4
%%%%%%%%%%%%%%%%%%%
\step{${d=4}$}\\\nopagebreak
$\begin{array}[t]{c | c c c} \hline
 \Spin_4  & e & (123) & (1234) \\
48 & 1 & 8 & 6 \\\hline
4+ & 2 & 1 & +\ii\sqrt{2} \\ 
4- & 2 & 1 & -\ii\sqrt{2} \\ 
3,1 & 4 & -1 & 0 \\\hline
\end{array}$&
$\begin{array}[t]{c | c c c} \hline
\ASpin_4 & e & (123) & (321) \\
24 & 1 & 4 & 4 \\\hline
4 & 2 & 1 & 1 \\
3,1+ & 2 & -\frac{1}{2}+\frac{1}{2}\ii\sqrt{3} & -\frac{1}{2}+\frac{1}{2}\ii\sqrt{3} \\ 
3,1- & 2 & -\frac{1}{2}-\frac{1}{2}\ii\sqrt{3} & -\frac{1}{2}-\frac{1}{2}\ii\sqrt{3} \\ \hline
\end{array}$ 
\\[-2pt]
$\begin{array}[t]{c | c c c} \hline
 \Sergeev_4 & e & C_{(123)} & C_{(1234),1}\\
768 & 1 & 32 & 48 \\\hline
4+ & 8  & 2 & +2\ii \\ 
4- & 8  & 2 & -2\ii \\
3,1 & 16 & -2 & 0 \\\hline
\end{array}$&
$\begin{array}[t]{c | c c c } \hline
 \Sergeev_4^0 & e & C_{(123)}(a) & C_{(123)}(b) \\
48 & 1 & 16 & 16 \\ \hline
4 & 8 & 2 & 2 \\
3,1+ & 8 & -1+\ii\sqrt{3}  &-1+\ii\sqrt{3} \\
3,1- & 8 & -1-\ii\sqrt{3} & -1-\ii\sqrt{3}\\ \hline
\end{array}$
\end{longtable}
%%%%%%%%%%%%%%%%%%%
%5
%%%%%%%%%%%%%%%%%%%
\begin{longtable}[l]{l l}
\step{${d=5}$}\\\nopagebreak
$\begin{array}[t]{c | c c c c c} \hline
\Spin_5  & e & (12345) & (123) & (1234) & (123)(45)  \\
240 & 1 & 24 & 20  & 30 & 20 \\\hline
5    & 4 & 1  & 2  & 0 & 0 \\ 
4,1+ & 6 & -1 & 0  & +\ii\sqrt{2}  & 0\\ 
4,1- & 6 & -1 & 0  & -\ii\sqrt{2}  & 0\\
3,2+ & 4 & 1  & -1 & 0 & +\ii\sqrt{3}\\ 
3,2- & 4 & 1  & -1 & 0 & -\ii\sqrt{3} \\\hline
\end{array}$\\[-2pt]
$\begin{array}[t]{c | c c c c} \hline
\ASpin_5 & e & (12345) & (54321) & (123)  \\
120 & 1 & 12 & 12 & 20 \\\hline
5+ & 2 & \frac{1}{2}+\frac{1}{2}\ii\sqrt{5} & \frac{1}{2}+\frac{1}{2}\ii\sqrt{5} & 1  \\
5- & 2 & \frac{1}{2}-\frac{1}{2}\ii\sqrt{5} & \frac{1}{2}-\frac{1}{2}\ii\sqrt{5} & 1 \\
4,1 & 6 & -1 & -1 & 0  \\ 
3,2 & 4 & 1  & 1  & -1 \\ \hline
\end{array}$ 
\\[-2pt]
$\begin{array}[t]{c | c c c c c} \hline
 \Sergeev_5 & e & C_{(12345)} & C_{(123)} & C_{(12345),1}\\
7680 & 1 & 384 & 80 & 384 \\\hline
5+ & 16 & 1 & 4 & +\ii\sqrt{5}\\
5- & 16 & 1 & 4 & -\ii\sqrt{5}\\
4,1 & 48 & -2 & 0 & 0\\
3,2 & 32 & 2 & -4 & 0\\ \hline
\end{array}$\\[-2pt]
$\begin{array}[t]{c | c c c c c} \hline
\Sergeev_5^0  & e & C_{(12345)} & C_{(123)} & C_{(1234)} & C_{(123)(45)}  \\
3840 & 1 & 384 & 80  & 240 & 160 \\\hline
5    & 16 & 1  & 4  & 0 & 0 \\ 
4,1+ & 24 & -1 & 0  & +2\ii  & 0\\ 
4,1- & 24 & -1 & 0  & -2\ii  & 0\\
3,2+ & 16 & 1  & -2 & 0 & +\ii\sqrt{6}\\ 
3,2- & 16 & 1  & -2 & 0 & -\ii\sqrt{6} \\\hline
\end{array}$
%%%%%%%%%%%%%%%%%%%
\end{longtable}

%\bibliographystyle{halpha}
%\bibliography{biblio}

\newcommand{\etalchar}[1]{$^{#1}$}

\end{document}